\documentclass{amsart}
\usepackage{amsmath, amsthm, amsfonts, amssymb, color}
\usepackage{mathrsfs}
\usepackage{enumerate}
\usepackage[colorlinks, citecolor=blue,pagebackref,hypertexnames=false]{hyperref}

\begin{document}

\newtheorem{theorem}{Theorem}[section]
\newtheorem{proposition}[theorem]{Proposition}
\newtheorem{coro}[theorem]{Corollary}
\newtheorem{lemma}[theorem]{Lemma}
\newtheorem{definition}[theorem]{Definition}
\newtheorem{example}[theorem]{Example}
\newtheorem{remark}[theorem]{Remark}
\newtheorem{fact}[theorem]{Fact}
\newtheorem{nota}[theorem]{Notation}
\renewcommand{\theequation}
{\thesection.\arabic{equation}}

\title[Spectral asymptotic formula]{Spectral asymptotic formula of Bessel--Riesz transform commutators}

\author{Zhijie Fan}
\address{Zhijie Fan, School of Mathematics and Information Science,
Guangzhou University, Guangzhou 510006, China}
\email{fanzhj3@mail2.sysu.edu.cn}

\author{Ji Li}
\address{Ji Li, School of Mathematical and Physical Sciences, Macquarie University, NSW 2109, Australia}
\email{ji.li@mq.edu.au}

\author{Fedor Sukochev}
\address{Fedor Sukochev, University of New South Wales, Kensington, NSW 2052, Australia}
\email{f.sukochev@unsw.edu.au}

\author{Dmitriy Zanin}
\address{Dmitriy Zanin, University of New South Wales, Kensington, NSW 2052, Australia}
\email{d.zanin@unsw.edu.au}

\date{\today}


\maketitle

\begin{abstract}
Let $R_{\lambda,j}$ be the $j$-th Bessel--Riesz transform, where $n\geq 1$, $\lambda>0$, and $j=1,\ldots,n+1$. In this article, we establish a Weyl type asymptotic for $[M_f,R_{\lambda,j}]$, the  commutator  of $R_{\lambda,j}$ with multiplication operator $M_f$, based on building a preliminary result that the endpoint weak Schatten norm of $[M_f,R_{\lambda,j}]$ can be characterised via homogeneous Sobolev norm $\dot{W}^{1,n+1}(\mathbb{R}_+^{n+1})$ of the symbol $f$. Specifically, the asymptotic coefficient  is equivalent to $\|f\|_{\dot{W}^{1,n+1}(\mathbb{R}_+^{n+1})}.$ Our main strategy is to relate Bessel--Riesz commutator to classical Riesz commutator via Schur multipliers, and then to establish the boundedness of Schur multipliers.
\end{abstract}

\tableofcontents

\section{Introduction and statement of main results}
\subsection{Background and our main results}
Singular values sequence is a fundamental object for studying the behavior of compact operators in infinite dimensional spaces, playing an important role in a variety of problems in many branches of analysis and mathematical quantum theory. Schatten and weak Schatten class norms, defined via certain norms of singular values sequence of compact operator, provide quantitative ways to measure the `size' and `decay' of compact operators on Hilbert space, respectively. The well-known connection with harmonic analysis is via the Riesz transform commutator, which  originates from the famous Nehari's \cite{Ne} problem and links closely with the quantized derivatives introduced by Connes \cite[IV]{Connes} in non-commutative geometry. The Schatten class properties of classical Riesz transform commutators were studied in \cite{JW,RS1988,RS}.

Let $\mathcal{L}_{p,q}(L_2(\mathbb{R}^{n}))$ (resp. $\mathcal{L}_{p}(L_2(\mathbb{R}^{n}))$) be the Schatten--Lorentz class (resp. Schatten class) whose definition will be postponed in Section \ref{Secpre}. In the classical Euclidean setting with $n\geq 2$, Janson--Wolff characterized the membership of $[R_j,M_f]$ (and thus, quantized derivatives of $f$) to the Schatten class $\mathcal{L}_p(L_2(\mathbb{R}^{n}))$ for $n<p<\infty$ in terms of membership of $f$ to the homogeneous Besov space $B_p(\mathbb{R}^n)$. Moreover, for the case $p\leq n$, they showed that $[R_j,M_f]\in \mathcal{L}_p(L_2(\mathbb{R}^{n}))$ if and only if $f$ is a constant. These results were extended to the Lorentz scale by Rochberg and Semmes \cite{RS}, with a characterisation provided in terms of oscillation spaces ${\rm OSC}_{p,q}(\mathbb{R}^n)$. The oscillation space is shown to degenerate into constant spaces when $0<p\leq n$ and $0<q\leq \infty$, except at the endpoint $(p,q)=(n,\infty)$. This is called the Janson--Wolff phenomenon (or cut-off point phenomenon), with somewhat different phenomena in the special dimension $n=1$ for the Hilbert transform commutator studied by Peller \cite{P,P2}. For the endpoint case $(p,q)=(n,\infty)$, Connes--Sullivan--Teleman \cite{CST} found and sketched a proof of the surprising equivalence between ${\rm OSC}_{n,\infty}(\mathbb{R}^n)$ and the homogeneous Sobolev space $\dot{W}^{1,n}(\mathbb{R}^{n})$. Lord--McDonald--Sukochev--Zanin \cite{LMSZ} (under a priori assumption $f\in L_\infty(\mathbb{R}^n)$) and Frank \cite{MR4716964} elaborated the proof of this identification with two completely different approach. These indicate that $[R_j,M_f]\in \mathcal{L}_{n,\infty}(L_2(\mathbb{R}^{n}))$ if and only if $f$ belongs to the homogeneous Sobolev space $\dot{W}^{1,n}(\mathbb{R}^{n})$.

Schatten  (and weak Schatten) class properties have been investigated for many significant operators in complex analysis and harmonic analysis, such as Szeg\H{o} projection \cite{FR}, big and little Hankel operators on the unit ball and  Heisenberg group \cite{FR},  big Hankel operator on Bergman space of the disk  \cite{AFP}, and  Hankel operators on the Bergman space of the unit ball \cite{MR1087805}, Riesz transform commutators on Heisenberg group \cite{MR4552581} and nilpotent Lie groups \cite{MR4756021}, and Riesz transform commutators associated with  Neumann--Laplacian operator \cite{Neumann} and associated with Bessel operator \cite{FLLXarxiv}, general Calder\'{o}n-Zygmund operators in the context of Euclidean space \cite{Zhangwei1,Zhangwei2} and in the setting of spaces of homogeneous type \cite{HRpreprint}. In particular, several equivalent endpoint $\mathcal{L}_{d,\infty}$ characterizations \cite{HRpreprint} for commutators were given via Haj\l asz type Sobolev spaces on spaces of homogeneous type with lower dimension $d$.

Bessel operator is an important operator arising in the study of complex analysis, harmonic analysis and partial differential equations. We refer the readers to \cite{MR64284,MS} for an elaboration of the background of Bessel operator. After the paper  \cite{FLLXarxiv} on weak Schatten estimate of Bessel--Riesz transform commutators,
Rupert L. Frank posed a question in March 2024 by personal communication to the authors:

``{\it can one derive a spectral asymptotic formula for the Bessel--Riesz transform

\hskip.19cm commutators}?''

The aim of this paper is to address this question. To be more explicit, for $\lambda> 0$, we consider the $(n+1)$-dimensional Bessel operator on $\mathbb{R}^{n+1}_+:=\mathbb{R}^n\times (0,\infty)$ ($n\geq 1$) from a celebrated work of Huber \cite{MR64284}, which is defined by
\begin{align}\label{Dlambda}
\Delta_\lambda:= -\sum_{k=1}^{n+1}\frac{\partial^2}{\partial x_k^2}-\frac{2\lambda}{x_{n+1}}\frac{\partial}{\partial x_{n+1}}.
\end{align}
The operator $\Delta_\lambda$ has a self-adjoint extension on the Hilbert space $L_2(\mathbb{R}_+^{n+1}, dm_\lambda)$ (See Section \ref{Selfexten}),  where $dm_\lambda(x):=x_{n+1}^{2\lambda}dx.$ Let $R_{\lambda,k}$ denote the $k$-th Bessel-Riesz transform
$$ R_{\lambda,k} := {\partial\over \partial x_k} \Delta_\lambda^{-{1\over2}},\quad k=1,\ldots,n+1. $$

As a preliminary step, we will first show that for the critical index $p=n+1$ (see \cite[Remark 1.6]{FLLXarxiv} for explanation about critical index),  the  endpoint weak Schatten class $\mathcal{L}_{n+1,\infty}(L_2(\mathbb{R}^{n+1}_+,m_{\lambda}))$ property of $[R_{\lambda,k},M_f]$  can be further characterized in terms the homogeneous Sobolev space norm of $f$, i.e., $\|\nabla f\|_{L_p(\mathbb{R}_+^{n+1})}$.

\begin{definition}
Let $\dot{W}^{1,p}(\mathbb{R}_+^{n+1})$ be the Sobolev space over $\mathbb{R}_+^{n+1},$ which consists of locally integrable functions $f$ whose distributional gradient belongs to $L_p(\mathbb{R}_+^{n+1})$. For such $f$, we set
$$\|f\|_{\dot{W}^{1,p}(\mathbb{R}_+^{n+1})}:=\|\nabla f\|_{L_p(\mathbb{R}_+^{n+1})}.$$
\end{definition}

The first result of our paper is as follows.

\begin{theorem}\label{main theorem} Let $n\geq1$, $\lambda>0$ and $1\leq k\leq n+1.$ For every $f\in L_{\infty}(\mathbb{R}^{n+1}_+),$
\begin{enumerate}[{\rm (i)}]
\item\label{mta} if $f\in\dot{W}^{1,n+1}(\mathbb{R}^{n+1}_+),$ then $[R_{\lambda,k},M_f]\in \mathcal{L}_{n+1,\infty}(L_2(\mathbb{R}^{n+1}_+,m_{\lambda}))$ and
$$\|[R_{\lambda,k},M_f]\|_{\mathcal{L}_{n+1,\infty}(L_2(\mathbb{R}^{n+1}_+,m_{\lambda}))}\leq C_{n,\lambda}\|f\|_{\dot{W}^{1,n+1}(\mathbb{R}^{n+1}_+)}.$$
\item\label{mtb} if $[R_{\lambda,k},M_f]\in \mathcal{L}_{n+1,\infty}(L_2(\mathbb{R}^{n+1}_+,m_{\lambda})),$ then $f\in\dot{W}^{1,n+1}(\mathbb{R}^{n+1}_+)$ and
$$\|f\|_{\dot{W}^{1,n+1}(\mathbb{R}^{n+1}_+)}\leq C_{n,\lambda}\|[R_{\lambda,k},M_f]\|_{\mathcal{L}_{n+1,\infty}(L_2(\mathbb{R}^{n+1}_+,m_{\lambda}))}.$$
\end{enumerate}
\end{theorem}

We now return to the question posed by Rupert L. Frank regarding the establishment of the spectral asymptotic formula. Before stating our result, we first briefly recall the background of this question. Motivated by Birman and Solomyak's work on eigenvalues of Schr\"{o}dinger-type operators \cite{MR1192782,MR279638}, the third and fourth authors and Frank \cite{FSZ} established the asymptotic of singular values for quantum derivatives (and thus, Riesz transform) in the classical setting. The spectral asymptotics are typically referred to as Weyl-type asymptotics. Moreover, the asymptotic coefficient happens to coincide with a non-asymptotic, uniform bound on the spectrum.

Our  main result is as follows.
\begin{theorem}\label{mtc} Let $n\geq1$, $\lambda>0$ and $1\leq k\leq n+1,$ then for every $f\in L_{\infty}(\mathbb{R}^{n+1}_+)\cap\dot{W}^{1,n+1}(\mathbb{R}^{n+1}_+),$ there exists a limit
$$\lim_{t\to\infty}t^{\frac1{n+1}}\mu_{B(L_2(\mathbb{R}^{n+1}_+,m_{\lambda}))}(t,[R_{\lambda,k},M_f])=C_{n,\lambda}\|f\|_{\dot{W}^{1,n+1}(\mathbb{R}_+^{n+1})}^{(k)}$$
for some constant $C_{n,\lambda}>0$ independent of $f.$ Here, $\|\cdot\|_{\dot{W}^{1,p}(\mathbb{R}_+^{n+1})}^{(k)},$ $1\leq k\leq n+1,$ is an equivalent semi-norm on $\dot{W}^{1,p}(\mathbb{R}_+^{n+1})$ given by the formula
$$\|f\|_{\dot{W}^{1,p}(\mathbb{R}_+^{n+1})}^{(k)}:=\Big(\int_{\mathbb{R}^{n+1}_+\times\mathbb{S}^n}\Big|(\partial_kf)(x)-s_k\sum_{m=1}^{n+1}s_m(\partial_mf)(x)\Big|^{p}dxds\Big)^{\frac1p}.$$
\end{theorem}

If $A\in\mathcal{L}_{1,\infty}$ is such that the limit $\lim_{t\to\infty}t\mu(t,A)$ exists, then the latter limit equals $\varphi(|A|)$ for every normalised continuous trace on $\mathcal{L}_{1,\infty}$ (we refer the reader to \cite{LSZold,LSZ1} for detailed study of singular traces). Applying this to $A=|[R_{\lambda,k},M_f]|^{n+1},$ we obtain the following trace formula.
\begin{coro}\label{connesintr}
Let $n\geq1$, $\lambda>0$ and $1\leq k\leq n+1,$ then there is a constant $C_{n,\lambda}>0$ such that for every $f\in L_{\infty}(\mathbb{R}^{n+1}_+)\cap\dot{W}^{1,n+1}(\mathbb{R}^{n+1}_+)$ and every continuous normalised trace $\varphi$ on $\mathcal{L}_{1,\infty}(L_2(\mathbb{R}^{n+1}_+))$, we have
$$\varphi(|[R_{\lambda,k},M_f]|^{n+1})=C_{n,\lambda}\|f\|_{\dot{W}^{1,n+1}(\mathbb{R}_+^{n+1})}^{(k)}.$$
\end{coro}
We remark that Corollary \ref{connesintr} is a non-trivial variant of Dixmier trace formula established in \cite[Theorem 2]{LMSZ} and \cite[Theorem 3 (iii)]{Connes1998}. It follows from  \cite[Theorem 3 (iii)]{Connes1998} that the Dixmier trace of the $d$th power of a quantised derivative on a compact Riemannian spin $d$-manifold can be represented as the integral of the $d$th power of the gradient of $f$. From Connes's perspective, by such a formula one can ``pass from quantized $1$-forms to ordinary forms, not by a classical limit, but by a direct application of the Dixmier trace". Extensions and analogies of this formula \cite[Theorem 3 (iii)]{Connes1998} in several setting have been established very recently (see \cite{LMSZ}, \cite{MR4654013}, \cite{MSX2019} and \cite{MSX2018} on the non-compact manifold $\mathbb{R}^d$, Heisenberg groups, the quantum Euclidean space and the quantum tori, respectively).

The following corollary is an almost immediate consequences of Theorem \ref{mtc}.
\begin{coro}
Let $n\geq1$, $\lambda>0$ and $1\leq k\leq n+1,$ if $f\in L_\infty(\mathbb{R}_+^{n+1})$ satisfies
$$\lim_{t\to\infty}t^{\frac1{n+1}}\mu_{B(L_2(\mathbb{R}^{n+1}_+,m_{\lambda}))}(t,[R_{\lambda,k},M_f])=0,$$
then $f$ is a constant.
\end{coro}

\subsection{Strategy and organisation of our paper}
The paper is organised as follows. Section \ref{Secpre} consists of three parts: in the first part, we give explanations of some notations which would be frequently used throughout the paper; in the second part, we recall the definitions and some basic properties of Schatten class, weak Schatten class and its separable part; in the third part, we provide a concrete self-adjoint extension of $\Delta_\lambda$ used throughout the text. In Section \ref{Sec3}, we provide the definition of Schur multiplier and establish an equality (see Proposition \ref{commutator representation lemma}) relating the Bessel--Riesz commutator to the classical Riesz commutator through several Schur multipliers. This allows us to reduce the problems in Theorem \ref{main theorem} to analogous problems in the classical setting and the study of these Schur multipliers. In Section \ref{Sec4}, we establish the boundedness of these Schur multipliers, which relies on a remarkable result established very recently in \cite[Theorem A]{MR4660138}. In Section \ref{upper bound section}, we provide a proof of upper bound in Theorem \ref{main theorem} by combining the useful formula given in Proposition \ref{commutator representation lemma} with the boundedness of Schur multipliers established in Section \ref{Sec4}. Section \ref{alternativepf} provides an alternative proof of upper bound in Theorem \ref{main theorem} for the case of $n\geq 2$. Avoiding the use of Proposition \ref{commutator representation lemma}, this approach is based on establishing corresponding Cwikel estimates in the Bessel setting, which are of independent interest in harmonic analysis and non-commutative geometry (see \cite{MR473576,LeSZ}). Sections \ref{lowbd} and \ref{spect} are devoted to showing lower bound in Theorem \ref{main theorem}, and Theorem \ref{mtc}, respectively.

\section{Preliminaries}\label{Secpre}
\setcounter{equation}{0}
\subsection{Notation}
Conventionally, we set $\mathbb{N}$ be the set of positive integers and $\mathbb{Z}_+:=\{0\}\cup\mathbb{N}$. For any multi-index $\alpha=(\alpha_1,\cdots,\alpha_{n+1})\in\mathbb{Z}_+^{n+1}$, we denote by $|\alpha|_1:=\sum_{j=1}^{n+1}\alpha_j.$ Let $e_k$ be the standard basis vector with $1$ in the $k$-th position and $0$ in all other positions. For $x\in\mathbb{R}$ and $k\in\mathbb{Z}_+$, we define the generalized binomial coefficient $\binom{x}{k}$ as follows:
$$\binom{x}{k}:=\frac{x(x-1)(x-2)\cdots (x-k+1)}{k!}.$$
For any positive integer $k$, we let $f^{(k)}$ be the $k$-th derivative of $f$. For any positive integer $k$ and $p\in [1,\infty]$, we let $\dot{W}^{k,p}(\Omega)$ and $W^{k,p}(\Omega)$ be the homogeneous Sobolev space and inhomogeneous Sobolev space over domain $\Omega\subset \mathbb{R}^{n+1}$, respectively. In our context, $\Omega$ maybe taken to be $\mathbb{R}_\pm$, $\mathbb{R}_+^{n+1}$ and $\mathbb{R}^{n+1}$.

We use the notation $X\circlearrowleft$ as an abbreviation to denote $X\rightarrow X$, mappings from $X$ to $X$. For a set $E \subset \mathbb{R}^{n+1}$, we denote by $\chi_E$ its characteristic function. For any $1\leq k,j\leq n+1$, $\delta_{k,j}$ is denoted by the Kronecker delta.

Throughout the whole paper, we denote by $C$ a positive constant which is independent of the main parameters, but it may
vary from line to line. We use $A\lesssim B$ to denote the statement that $A\leq CB$ for some constant $C>0$ which is independent of the main parameters. Let $B(x,r)$ be the ball in $\mathbb{R}^{n+1}$ centered at $x$ with radius $r>0$.

{\bf Convention}: {\it In the sequel, unless  otherwise specific, we always assume that $\lambda>0$ and $n\in \mathbb{N}$.}
\subsection{Schatten class}\label{Spdef}
In this subsection, for the convenience of readers, we collect some standard material about Schatten class from literatures. Let $B(H)$ be a space consisting of all bounded linear operators on $H$, equipped with the uniform operator norm $\|\cdot\|_\infty$.  Note that if $T$ is any compact operator on $H$, then $T^{*}T$ is compact, positive and therefore diagonalizable. We define the singular values $\{\mu_{B(H)}(k,T)\}_{k=0}^{\infty}$ be the sequence of square roots of eigenvalues of $T^{*}T$ (counted according to multiplicity). Equivalently, $\mu_{B(H)}(k,T)$ can be characterized by
\begin{align*}
\mu_{B(H)}(k,T)=\inf\{\|T-F\|_\infty:{\rm rank}(F)\leq k\}.
\end{align*}
The formula above can be extended to define a singular value function $\mu_{B(H)}(T)$ for $t>0$ as follows:
\begin{align*}
\mu_{B(H)}(t,T):=\inf\{\|T-F\|_\infty:{\rm rank}(F)\leq t\}.
\end{align*}
The singular value function $\mu(f)$ of measurable function $f$ on $\mathbb{R}^{n+1}$, also called the decreasing rearrangement of $f$, is defined by
$$\mu(t,f):=\inf\{\|f\cdot (1-\chi_F)\|_\infty:\ m(F)\leq t\},\quad t>0.$$
The space of measurable functions $f$ such that $\mu(t,f)$ is finite for all $t>0$ is denoted by $S(\mathbb{R}^{n+1})$. For $f,g\in S(\mathbb{R}^{n+1}),$ then we say that $g$ is submajorized by $f$, denoted by $g\prec\prec f$, if
$$\int_0^t \mu(s,g)ds\leq \int_0^t \mu(s,f)ds,\quad t>0.$$
Moreover, for $T\in B(H)$ and $f\in S(\mathbb{R}^{n+1})$, we write $T\prec\prec f$ if $\mu_{B(H)}(T)\prec\prec \mu(f)$ (see \cite[Section 1]{LSZ2} for more explanation about submajorization).

A sequence $a=\{a_k\}_{k\geq0}$ is in the $\ell_{p,q}$ for some $0<p<\infty$ and $0<q<\infty$ provided the decreasing rearrangement $\mu(a)$ satisfies
$$\sum_{k\geq0 }\mu(k,a)^q(k+1)^{q/p-1}<\infty.$$
For the endpoint $q=\infty$, a sequence $a=\{a_k\}_{k\geq0}$ is in the $\ell_{p,\infty}$ for some $0<p<\infty$ if
$$\sup_{k\geq0}(k+1)^{1/p}\mu(k,a)<\infty.$$
See also \cite[Proposition 1.4.9]{MR3243734} for equivalent definitions. For $0<p<\infty$ and $0<q\leq \infty$, we say that $T\in \mathcal{L}_{p,q}(H)$ if $\{\mu_{B(H)}(k,T)\}\in \ell_{p,q}$. In particular, for $0<p<\infty$, a compact operator $T$ on $H$ is said to belong to the Schatten class $\mathcal{L}_{p}(H)$ if $\{\mu_{B(H)}(k,T)\}_{k=0}^{\infty}$ is $p$-summable, i.e. in the sequence space $\ell_{p}$. If $p\geq 1$, then the $\mathcal{L}_{p}(H)$ norm is defined as
\begin{align*}
\|T\|_{\mathcal{L}_{p}(H)}:=\left(\sum_{k=0}^{\infty}\mu_{B(H)}(k,T)^{p}\right)^{1/p}.
\end{align*}
With this norm $\mathcal{L}_p(H)$ is a Banach space and an ideal of $B(H)$. Analogously, a compact operator $T$ on $H$ is said to belong to the weak Schatten class  $\mathcal{L}_{p,\infty}(H)$ if $\{\mu_{B(H)}(k,T)\}_{k=0}^{\infty}$ is in $\ell_{p,\infty}$, with quasi-norm defined as:
\begin{align*}
\|T\|_{\mathcal{L}_{p,\infty}(H)}:=\sup\limits_{k\geq 0}(k+1)^{\frac 1p}\mu_{B(H)}(k,T)<\infty.
\end{align*}
More details about Schatten class can be found in e.g. \cite{LSZold,LSZ1}.

For $1\leq p <\infty$, we define an ideal $(\mathcal{L}_{p,\infty}(H))_{0}$ by setting
\begin{align*}
(\mathcal{L}_{p,\infty}(H))_{0}:=\{T \in \mathcal{L}_{p,\infty}(H)   :  \lim_{t\to +\infty} t^{\frac 1p} \mu_{B(H)}(t,T) = 0 \}.
\end{align*}
Equivalently, $(\mathcal{L}_{p,\infty}(H))_{0}$ is the closure of the ideal of all finite rank operators in the norm $\|\cdot\|_{\mathcal{L}_{p,\infty}(H)}$.
As a closed subspace of $\mathcal{L}_{p,\infty}(H)$, this ideal is commonly called the separable part of $\mathcal{L}_{p,\infty}(H)$ (See \cite{LSZold,LSZ1} for more details about separable part).

Recall from \cite[Corollary 2.3.16]{LSZold} that for two compact operators $T$ and $G$,
$$\mu_{B(H)}(t,T+G)\leq \mu_{B(H)}(\frac t2,T)+\mu_{B(H)}(\frac t2,G).$$
This implies that
$$\|T+G\|_{\mathcal{L}_{p,\infty}(H)}\leq 2^{\frac1p}(\|T\|_{\mathcal{L}_{p,\infty}(H)}+\|G\|_{\mathcal{L}_{p,\infty}(H)}).$$
Moreover, it also implies the following ideal property of $(\mathcal{L}_{p,\infty}(H))_0$: if $A\in \mathcal{L}_{p,\infty}(H)$ and $B\in (\mathcal{L}_{q,\infty}(H))_0$, then $AB\in (\mathcal{L}_{r,\infty}(H))_{0}$, where $1\leq p,q,r\leq\infty$ satisfying $1/r=1/p+1/q$.

The following Lemma is useful to transform the problems on weighted $L_2$ space into problems on unweighted $L_2$ space.
\begin{lemma}\label{weight}
For every $1\leq p<\infty$, we have
\begin{enumerate}[{\rm (i)}]
\item $T\in\mathcal{L}_{p}(L_2(\mathbb{R}_+^{n+1},m_\lambda))$ if and only if $T\in\mathcal{L}_{p}(L_2(\mathbb{R}_+^{n+1}))$. Moreover,
$$\|T\|_{\mathcal{L}_{p}(L_2(\mathbb{R}_+^{n+1},m_\lambda))}=\|T\|_ {\mathcal{L}_{p}(L_2(\mathbb{R}_+^{n+1}))};$$
\item $T\in\mathcal{L}_{p,\infty}(L_2(\mathbb{R}_+^{n+1},m_\lambda))$ if and only if $T\in\mathcal{L}_{p,\infty}(L_2(\mathbb{R}_+^{n+1}))$. Moreover,
$$\|T\|_{\mathcal{L}_{p,\infty}(L_2(\mathbb{R}_+^{n+1},m_\lambda))}=\|T\|_ {\mathcal{L}_{p,\infty}(L_2(\mathbb{R}_+^{n+1}))};$$
\item
$$\lim_{t\to\infty}t^{\frac1{n+1}}\mu_{B(L_2(\mathbb{R}^{n+1}_+,m_{\lambda}))}(t,T)=\lim_{t\to\infty}t^{\frac1{n+1}}\mu_{B(L_2(\mathbb{R}^{n+1}_+))}(t,M_{x_{n+1}^\lambda}TM_{x_{n+1}^{-\lambda}}).$$
In particular, $T\in(\mathcal{L}_{p,\infty}(L_2(\mathbb{R}_+^{n+1},m_\lambda)))_0$ if and only if $M_{x_{n+1}^\lambda}TM_{x_{n+1}^{-\lambda}}\in(\mathcal{L}_{p,\infty}(L_2(\mathbb{R}_+^{n+1})))_0$.
\end{enumerate}
\end{lemma}
\begin{proof}
The argument can be referred to \cite[Lemma 2.7]{MR4706933}. The key observation is that $M_{x_{n+1}^{-\lambda}}$ is a unitary operator from $L_2(\mathbb{R}_+^{n+1})$ to $L_2(\mathbb{R}_+^{n+1},m_\lambda)$. This implies
$$\|T\|_{\mathcal{L}_{p}(L_2(\mathbb{R}_+^{n+1},m_\lambda))}=\|M_{x_{n+1}^\lambda}TM_{x_{n+1}^{-\lambda}}\|_ {\mathcal{L}_{p}(L_2(\mathbb{R}_+^{n+1}))}=\|T\|_ {\mathcal{L}_{p}(L_2(\mathbb{R}_+^{n+1}))},$$
$$\|T\|_{\mathcal{L}_{p,\infty}(L_2(\mathbb{R}_+^{n+1},m_\lambda))}=\|M_{x_{n+1}^\lambda}TM_{x_{n+1}^{-\lambda}}\|_ {\mathcal{L}_{p,\infty}(L_2(\mathbb{R}_+^{n+1}))}=\|T\|_ {\mathcal{L}_{p,\infty}(L_2(\mathbb{R}_+^{n+1}))},$$
and the third statement.
\end{proof}

Let $E$ be the zero extension operator defined by
$$(Ef)(x):=\begin{cases}
f(x),& x\in \mathbb{R}_+^{n+1},\\
0,& x\notin \mathbb{R}_+^{n+1},
\end{cases}
$$
for every function $f$ on $\mathbb{R}_+^{n+1}.$ It is immediate that $E^{\ast}E=1$ and $EE^{\ast}=M_{\chi_{\mathbb{R}_+^{n+1}}}.$

The following Lemma is useful to transform the problems on $\mathbb{R}_+^{n+1}$ into problems on $\mathbb{R}^{n+1}$.
\begin{lemma}\label{half}
For every $1\leq p<\infty$, we have
\begin{enumerate}[{\rm (i)}]
\item $T\in\mathcal{L}_{p}(L_2(\mathbb{R}_+^{n+1}))$ if and only if $ETE^{\ast}\in\mathcal{L}_{p}(L_2(\mathbb{R}^{n+1}))$. Moreover,
$$\|T\|_{\mathcal{L}_{p}(L_2(\mathbb{R}_+^{n+1}))}=\|ETE^{\ast}\|_{\mathcal{L}_{p}(L_2(\mathbb{R}^{n+1}))};$$
\item $T\in\mathcal{L}_{p,\infty}(L_2(\mathbb{R}_+^{n+1}))$ if and only if $ETE^{\ast}\in\mathcal{L}_{p,\infty}(L_2(\mathbb{R}^{n+1}))$. Moreover,
$$\|T\|_{\mathcal{L}_{p,\infty}(L_2(\mathbb{R}_+^{n+1}))}=\|ETE^{\ast}\|_{\mathcal{L}_{p,\infty}(L_2(\mathbb{R}^{n+1}))};$$
\item
$$\lim_{t\to\infty}t^{\frac1{n+1}}\mu_{B(L_2(\mathbb{R}_+^{n+1}))}(t,T)=\lim_{t\to\infty}t^{\frac1{n+1}}\mu_{B(L_2(\mathbb{R}^{n+1}))}(t,ETE^{\ast}).$$
In particular, $T\in(\mathcal{L}_{p,\infty}(L_2(\mathbb{R}_+^{n+1})))_0$ if and only if $ETE^{\ast}\in(\mathcal{L}_{p,\infty}(L_2(\mathbb{R}^{n+1})))_0.$
\end{enumerate}
\end{lemma}
Let $\Delta$ be the standard non-negative Laplacian operator on $\mathbb{R}^{n+1}$, which is given by $\Delta:=-\sum_{k=1}^{n+1}\partial_{k}^2$.  The following Lemma is the Cwikel estimate for $M_fg(\sqrt{\Delta})$.
\begin{lemma}\cite[Theorem 4.2]{MR541149}\label{cwikel estimate in Euclidean setting}
Let $p>2$ and $n\geq 2.$ Then there is a constant $C_{n,p}>0$ such that for any $f\in L_p(\mathbb{R}^{n+1}),$ $g\in L_{p,\infty}(\mathbb{R}_+,r^ndr),$ we have
$$\|M_fg(\sqrt{\Delta})\|_{\mathcal{L}_{p,\infty}(L_2(\mathbb{R}^{n+1}))}\leq C_{n,p}\|f\|_{L_p(\mathbb{R}^{n+1})}\|g\|_{L_{p,\infty}(\mathbb{R}_+,r^ndr)}.$$
\end{lemma}
\subsection{Self-adjoint extension of $\Delta_\lambda$}\label{Selfexten}

Let $\mathcal{F}(f)$ denote the Fourier transform of a function $f\in C_c^\infty(\mathbb{R}^n)$ by the formula
\begin{align*}
\mathcal{F}(f)(\xi):=\frac{1}{(2\pi)^{n/2}}\int_{\mathbb{R}^{n}}f(x)e^{-i\langle x,\xi \rangle}dx, \ \ \xi\in\mathbb{R}^{n},
\end{align*}
where $\langle \cdot,\cdot \rangle$ denotes the inner product in $\mathbb{R}^n$.
Moreover, let $\mathcal{F}^{-1}$ be the inverse Fourier transform of $f$, which is given by $\mathcal{F}^{-1}(f)(\xi):=\mathcal{F}(f)(-\xi)$.

For $\lambda>0,$ let $J_\lambda$ denote the Bessel function of order $\lambda$, and let $U_{\lambda}(f)$ denote the Fourier-Bessel transform of a function $f\in C^{\infty}_c(\mathbb{R}_+)$ by the formula
$$(U_{\lambda}f)(x):=\int_{\mathbb{R}_+}f(y)\phi_\lambda(xy)y^{2\lambda}dy,\quad f\in C^{\infty}_c(\mathbb{R}_+),$$
where
$$\phi_\lambda(\xi):=\xi^{\frac12-\lambda}J_{\lambda-\frac12}(\xi),\quad x>0.$$
It is known that $U_{\lambda}$ extends to a unitary operator on the Hilbert space $L_2(\mathbb{R}_+,x^{2\lambda}dx)$ (see \cite{MS} and \cite[Proposition 2.5]{MTaylor}). The following inversion formula holds:
$$f(x)=\int_{\mathbb{R}_+} (U_{\lambda}f)(y)\phi_{\lambda}(xy)y^{2\lambda}dy.$$

It is an elementary fact that
$$(\mathcal{F}\otimes U_{\lambda})\Delta_{\lambda}f=M_{|x|^2}(\mathcal{F}\otimes U_{\lambda})f,\quad f\in C^{\infty}_c(\mathbb{R}_+^{n+1}).$$
Hence, there exists a self-adjoint extension of $\Delta_{\lambda}$ given by the formula
$$\Delta_{\lambda}=(\mathcal{F}\otimes U_{\lambda})^{-1}M_{|x|^2}(\mathcal{F}\otimes U_{\lambda}).$$
In what follows, $\Delta_{\lambda}$ denotes exactly this extension.
\section{A connection between two Riesz transforms commutators}\label{Sec3}
In this section, we relate the Bessel--Riesz transforms commutators to the classical Riesz transforms commutators via Schur multipliers. This allows to reduce the problems in Theorem \ref{main theorem} to that in the classical setting.
\begin{lemma}\label{kernel of sqrt delta} Integral kernel of $\Delta_{\lambda}^{-\frac12}$ is of the shape
$$K_{\Delta_{\lambda}^{-\frac12}}(x,y)=\kappa_{n,\lambda}^{[1]}\int_0^2Q_t^{-\lambda-\frac{n}{2}}(x,y)(2t-t^2)^{\lambda-1}dt,$$
where $\kappa_{n,\lambda}^{[1]}:=\frac{2^n}{\sqrt{\pi}}\frac{\Gamma(\lambda+\frac{1}{2})}{\Gamma(\lambda)}$ and
$$Q_t(x,y):=|x-y|^2+2tx_{n+1}y_{n+1},\quad x,y\in\mathbb{R}^{n+1}_+.$$
\end{lemma}
\begin{proof} Recall from \cite[formula (2.11)]{FLLXarxiv} that integral kernel of $e^{-t^2\Delta_{\lambda}}$ is of the shape
$$K_{e^{-s^2\Delta_{\lambda}}}(x,y)=\kappa_{\lambda}s^{-2\lambda-1-n}\int_0^{\pi}\exp(-\frac{Q_{1-\cos(\theta)}(x,y)}{4s^2})\sin^{2\lambda-1}(\theta)d\theta,$$
where $\kappa_\lambda:=\frac{1}{2^{2\lambda}\sqrt{\pi}}\frac{\Gamma(\lambda+\frac{1}{2})}{\Gamma(\lambda)}$. By functional calculus \cite{MR1336382},
$$\Delta_{\lambda}^{-\frac12}=\frac{2}{\sqrt{\pi}}\int_0^{\infty}e^{-s^2\Delta_{\lambda}}ds,$$
where the integral converges in the strong operator topology of $B(L_2(\mathbb{R}_+^{n+1},m_\lambda))$.
Hence, integral kernel of $\Delta_{\lambda}^{-\frac12}$ is of the shape
\begin{align*}
&K_{\Delta_{\lambda}^{-\frac12}}(x,y)\\
&=\frac{2}{\sqrt{\pi}}\int_0^{\infty}\Big(\kappa_{\lambda}s^{-2\lambda-1-n}\int_0^{\pi}\exp(-\frac{Q_{1-\cos(\theta)}(x,y)}{4s^2})\sin^{2\lambda-1}(\theta)d\theta\Big)ds\\
&=\frac{2}{\sqrt{\pi}}\kappa_{\lambda}\int_0^{\pi}\sin^{2\lambda-1}(\theta)\Big(\int_0^{\infty}s^{-2\lambda-1-n}\exp(-\frac{Q_{1-\cos(\theta)}(x,y)}{4s^2})ds\Big)d\theta\\
&\stackrel{s=u^{-\frac12}}{=}\frac{2}{\sqrt{\pi}}\kappa_{\lambda}\int_0^{\pi}\sin^{2\lambda-1}(\theta)\Big(\int_0^{\infty}u^{\lambda+\frac{n+1}{2}}\exp(-\frac{Q_{1-\cos(\theta)}(x,y)}{4}u)\frac{du}{2u^{\frac32}}\Big)d\theta\\
&=\frac1{\sqrt{\pi}}\kappa_{\lambda}\int_0^{\pi}\sin^{2\lambda-1}(\theta)\Big(\int_0^{\infty}u^{\lambda+\frac{n}{2}-1}\exp(-\frac{Q_{1-\cos(\theta)}(x,y)}{4}u)du\Big)d\theta\\
&=\frac{4^{\lambda+\frac{n}{2}}}{\sqrt{\pi}}\kappa_{\lambda}\Gamma(\lambda+\frac{n}{2})\int_0^{\pi}Q_{1-\cos(\theta)}^{-\lambda-\frac{n}{2}}(x,y)\sin^{2\lambda-1}(\theta)d\theta\\
&\stackrel{t=1-\cos(\theta)}{=}\kappa_{n,\lambda}^{[1]}\int_0^2Q_t^{-\lambda-\frac{n}{2}}(x,y)(2t-t^2)^{\lambda-1}dt.
\end{align*}
This ends the proof of Lemma \ref{kernel of sqrt delta}.
\end{proof}

We frequently use the following functions, which link Bessel--Riesz transforms with the classical Riesz transforms.
\begin{nota}\label{fkl notation} For $k,l\in\mathbb{Z}_+,$ denote
$$F_{k,l}(x):=x^{n+k}\int_0^2(x^2+2t)^{-\lambda-\frac{n}{2}-1}(2t-t^2)^{\lambda-1}t^ldt,\quad x\in(0,\infty).$$	
\end{nota}
\begin{nota} For $1\leq k\leq n+1,$ denote
$$H(x,y):=\frac{|x-y|}{|x_{n+1}y_{n+1}|^{\frac12}},\quad x,y\in\mathbb{R}^{n+1},$$
$$K_k(x,y):=\frac{(x-y)_k}{|x-y|^{n+2}(x_{n+1}y_{n+1})^{\lambda}},\quad x,y\in\mathbb{R}^{n+1}_+,$$	
$$h_k(x,y):=\frac{(x-y)_k}{|x-y|},\quad x,y\in\mathbb{R}^{n+1},$$
$$a(x,y):=\bigg(\frac{\min\{y_{n+1},x_{n+1}\}}{\max\{y_{n+1},x_{n+1}\}}\bigg)^{\frac12},\quad x,y\in\mathbb{R}^{n+1}_+,$$
$$b(x,y):=\chi_{\{x_{n+1}<y_{n+1}\}},\quad x,y\in\mathbb{R}^{n+1}_+.$$
\end{nota}
With these notations, we first relate the integral kernel of the Bessel--Riesz transforms to that of the classical Riesz transforms.
\begin{lemma}\label{Bessel--Riesz kernel representation lemma} Integral kernel of $R_{\lambda,k},$ $1\leq k\leq n+1,$ is given by the formula
\begin{align*}
K_{R_{\lambda,k}}&=-\kappa^{[2]}_{n,\lambda}\cdot\Big((F_{2,0}\circ H)\cdot K_k+\delta_{k,n+1}\sum_{l=1}^{n+1}a\cdot h_l\cdot (F_{1,1}\circ H)\cdot K_l\\
&\hspace{1.0cm}-\delta_{k,n+1}\sum_{l=1}^{n+1}b\cdot h_{n+1}\cdot h_l\cdot (F_{2,1}\circ H)\cdot K_l\Big),\nonumber
\end{align*}
where $\kappa^{[2]}_{n,\lambda}:=(2\lambda+n)\kappa^{[1]}_{n,\lambda}.$
\end{lemma}
\begin{proof} Differentiating the right hand side in Lemma \ref{kernel of sqrt delta} by $x_{k},$ we obtain
\begin{align*}
K_{R_{\lambda,k}}(x,y)&=-(2\lambda+n)\kappa^{[1]}_{n,\lambda}(x-y)_{k}\int_0^2Q_t^{-\lambda-\frac{n}{2}-1}(x,y)(2t-t^2)^{\lambda-1}dt\\
&\hspace{1.0cm}-(2\lambda+n)\kappa^{[1]}_{n,\lambda}\delta_{k,n+1}y_{n+1}\int_0^2Q_t^{-\lambda-\frac{n}{2}-1}(x,y)(2t-t^2)^{\lambda-1}tdt.
\end{align*}
It follows from the definitions of $F_{2,0}$ and $F_{1,1}$ that
$$\int_0^2Q_t^{-\lambda-\frac{n}{2}-1}(x,y)(2t-t^2)^{\lambda-1}dt=\frac1{|x-y|^{n+2}(x_{n+1}y_{n+1})^{\lambda}}\cdot F_{2,0}(\frac{|x-y|}{(x_{n+1}y_{n+1})^{\frac12}}),$$
and that
\begin{align*}
&y_{n+1}\int_0^2Q_t^{-\lambda-\frac{n}{2}-1}(x,y)(2t-t^2)^{\lambda-1}tdt\\
&=\frac1{|x-y|^{n+1}(x_{n+1}y_{n+1})^{\lambda}}\cdot (\frac{y_{n+1}}{x_{n+1}})^{\frac12}F_{1,1}(\frac{|x-y|}{(x_{n+1}y_{n+1})^{\frac12}}).
\end{align*}
Clearly,
$$(\frac{y_{n+1}}{x_{n+1}})^{\frac12}=a(x,y)-b(x,y)h_{n+1}(x,y)H(x,y).$$
Thus,
\begin{align*}
K_{R_{\lambda,k}}(x,y)&=-\kappa^{[2]}_{n,\lambda}K_k(x,y)\cdot (F_{2,0}\circ H)(x,y)\\
&\hspace{-1.0cm}-\delta_{k,n+1}\kappa^{[2]}_{n,\lambda}\frac1{|x-y|^{n+1}(x_{n+1}y_{n+1})^{\lambda}}\cdot a(x,y)\cdot (F_{1,1}\circ H)(x,y)\\
&\hspace{-1.0cm}+\delta_{k,n+1}\kappa^{[2]}_{n,\lambda}\frac1{|x-y|^{n+1}(x_{n+1}y_{n+1})^{\lambda}}\cdot b(x,y)\cdot h_{n+1}(x,y)\cdot (F_{2,1}\circ H)(x,y).
\end{align*}
This, in combination with the equality
$$\frac1{|x-y|^{n+1}(x_{n+1}y_{n+1})^{\lambda}}=\sum_{l=1}^{n+1}K_l(x,y)\cdot h_l(x,y),$$
completes the proof of Lemma.
\end{proof}
We recall the definition of Schur multipliers. Let $(X,\mu)$ be a $\sigma$-finite measure space. In the sequel, $(X,\mu)$ would be frequently taken to be $(\mathbb{R}_+^{n+1},m_\lambda)$ or $(\mathbb{R}^{n+1},dx)$. Recall that $\mathcal{L}_2(L_2(X,\mu))\backsimeq L_2(X\times X,\mu\times\mu)$ via the identification $T\rightarrow (K_T(x,y))_{x,y\in X},$ where $K_T(x,y)\in L_2(X\times X,\mu\times\mu)$ is the integral kernel of $T.$

\begin{definition}
Let $M:X\times X\rightarrow \mathbb{C}$ be a bounded function, then the Schur multiplier $\mathfrak{S}_M: \mathcal{L}_2(X,\mu)\rightarrow \mathcal{L}_2(X,\mu)$ is a linear operator defined by sending an operator $T:L_2(X,\mu)\rightarrow L_2(X,\mu)$ with kernel $K_T(x,y)$ to the operator with integral kernel $M(x,y)K_T(x,y).$
\end{definition}
For $1\leq k\leq n+1$, we define $k$-th classical Riesz transform by $R_k:=\partial_k\Delta^{-\frac12}$.
Next, we apply Schur multipliers and the formula given in Lemma \ref{Bessel--Riesz kernel representation lemma} to relate Bessel--Riesz transforms commutators to the classical Riesz transforms commutators, which is the main result of this section and can be presented as follows.
\begin{proposition}\label{commutator representation lemma} There exists a constant $\kappa_{n,\lambda}^{[3]}$ such that for every $f\in C^{\infty}_c(\mathbb{R}^{n+1})$ and for every $1\leq k\leq n+1,$
\begin{align*}
[R_{\lambda,k},M_{E^{\ast}f}]&=\kappa_{n,\lambda}^{[3]}\mathfrak{S}_{F_{2,0}\circ H}\Big( M_{x_{n+1}^{-\lambda}} E^{\ast}[R_k,M_{f}]E M_{x_{n+1}^{\lambda}}\Big)\\
&\hspace{-1.0cm}+\kappa_{n,\lambda}^{[3]}\delta_{k,n+1}\sum_{l=1}^{n+1}\Big(\mathfrak{S}_{h_l}\circ\mathfrak{S}_a\circ\mathfrak{S}_{F_{1,1}\circ H}\Big)\Big(M_{x_{n+1}^{-\lambda}} E^{\ast}[R_l,M_{f}]E M_{x_{n+1}^{\lambda}}\Big)\nonumber\\
&\hspace{-1.0cm}-\kappa_{n,\lambda}^{[3]}\delta_{k,n+1}\sum_{l=1}^{n+1}\Big(\mathfrak{S}_{h_l}\circ\mathfrak{S}_{h_{n+1}}\circ\mathfrak{S}_b\circ\mathfrak{S}_{F_{2,1}\circ H}\Big)\Big( M_{x_{n+1}^{-\lambda}} E^{\ast}[R_l,M_{f}]E M_{x_{n+1}^{\lambda}}\Big).\nonumber
\end{align*}
\end{proposition}
\begin{proof} Note that the integral kernel of the operator $R_l$ is
$$(x,y)\to\omega_n\frac{(y-x)_l}{|x-y|^{n+2}},\quad x,y\in\mathbb{R}^{n+1},$$
for some constant $\omega_{n}>0.$ Setting $\kappa_{n,\lambda}^{[3]}:=\omega_n^{-1}\kappa_{n,\lambda}^{[2]},$ we infer the assertion from
Lemma \ref{Bessel--Riesz kernel representation lemma}.
\end{proof}

\section{Boundedness of Schur multipliers}\label{Sec4}
\setcounter{equation}{0}

This section is devoted to establishing the boundedness of Schur multipliers associated with symbols $F_{k,l}\circ H$, $a$, $b$ and $h_m$ appearing in Proposition \ref{commutator representation lemma},  where $(k,l)\in\{(2,0),(1,1),(2,1)\}$ and $1\leq m\leq n+1$. The main results of this section are formulated in Propositions \ref{fkl schur lemma} and \ref{standard schur lemma}.

In order to show the boundedness of Schur multipliers associated with symbol $F_{k,l}\circ H$, we will first establish an abstract criterion, which can be stated as follows.
\begin{theorem}\label{main schur theorem} Let $F$ be a smooth function on $(0,\infty)$ and right continuous at $0.$ If $F\circ\exp\in W^{1+\lceil\frac{n+1}{2}\rceil,2}(\mathbb{R}_+)$ and $(F-F(0))\circ\exp\in W^{1+\lceil\frac{n+1}{2}\rceil,2}(\mathbb{R}_-),$ then $\mathfrak{S}_{F\circ H}$ is a bounded mapping from $\mathcal{L}_p(L_2(\mathbb{R}^{n+1}))$ to itself for every $1<p<\infty.$
\end{theorem}

The proof of Theorem \ref{main schur theorem} relies on a landmark result established very recently by Conde-Alonso etc., which provides a simple sufficient condition for the $\mathcal{L}_p$-boundedness of Schur multipliers.
\begin{theorem}\cite[Theorem A]{MR4660138}\label{Annalspaper}
Let $M\in C^{[\frac{n}{2}]+1}((\mathbb{R}^n\times\mathbb{R}^n)\backslash\{x=y\})$ be a bounded function satisfying the condition
$$\left|\frac{\partial^\alpha M}{\partial x^\alpha}(x,y)\right|+\left|\frac{\partial^\alpha M}{\partial y^\alpha}(x,y)\right|\leq C_\alpha |x-y|^{-|\alpha|_1},\quad x,y\in\mathbb{R}^n,$$
for each $\alpha\in\mathbb{Z}_+^n$ with $|\alpha|_1\leq \lceil\frac{n}{2}\rceil+1.$ Under these assumptions, $\mathfrak{S}_M$ is bounded on $\mathcal{L}_p(L_2(\mathbb{R}^n))$ for every $1<p<\infty.$
\end{theorem}

\begin{lemma}\label{ratio is bounded}
Let $(x,y)=(x',x_{n+1},y',y_{n+1})\in\mathbb{R}^{n+1}\times \mathbb{R}^{n+1}$ be any point such that $H(x,y)\leq 1,$ then
$$\frac{3-\sqrt{5}}{2}|y_{n+1}|\leq |x_{n+1}|\leq \frac{3+\sqrt{5}}{2}|y_{n+1}|.$$	
\end{lemma}
\begin{proof} Using the obvious inclusion
$$\left\{(x_{n+1},y_{n+1})\in\mathbb{R}^2:\ \frac{|x-y|^2}{|x_{n+1}y_{n+1}|}\leq 1\right\}\subset \left\{(x_{n+1},y_{n+1})\in\mathbb{R}^2:\ \frac{|x_{n+1}-y_{n+1}|^2}{|x_{n+1}y_{n+1}|}\leq 1\right\},$$
we deduce that
\begin{align*}
\inf\left\{\frac{|x_{n+1}|}{|y_{n+1}|}:\ \frac{|x-y|^2}{|x_{n+1}y_{n+1}|}\leq 1\right\}&\geq\inf\left\{\frac{|x_{n+1}|}{|y_{n+1}|}:\ \frac{|x_{n+1}-y_{n+1}|^2}{|x_{n+1}y_{n+1}|}\leq 1\right\}\\
&=\inf\{u>0:\ u+u^{-1}\leq 3\}=\frac{3-\sqrt{5}}{2},
\end{align*}
and that
\begin{align*}
\sup\left\{\frac{|x_{n+1}|}{|y_{n+1}|}:\ \frac{|x-y|^2}{|x_{n+1}y_{n+1}|}\leq 1\right\}&\leq\sup\left\{\frac{|x_{n+1}|}{|y_{n+1}|}:\ \frac{|x_{n+1}-y_{n+1}|^2}{|x_{n+1}y_{n+1}|}\leq 1\right\}\\
&=\sup\{u>0:\ u+u^{-1}\leq 3\}=\frac{3+\sqrt{5}}{2}.
\end{align*}
This ends the proof of Lemma \ref{ratio is bounded}.
\end{proof}

\begin{lemma}\label{H derivatives lemma} For every multi-index $\alpha\in\mathbb{Z}^{2n+2}_+,$ there is a constant $c_\alpha>0$ such that for any $(x,y)\in \{x,y\in\mathbb{R}^{n+1}:\ H(x,y)\leq 1\}$,
$$|(\partial_{x,y}^{\alpha}H)(x,y)|\leq\frac{c_{\alpha}}{|x-y|^{\alpha}},\quad x,y\in\mathbb{R}^{n+1}.$$
\end{lemma}
\begin{proof} To lighten the notations, we may assume without loss of generality that $x_{n+1},y_{n+1}>0.$

Set $H_1(x,y)=|x-y|,$ $H_2(x,y)=x_{n+1}^{-\frac12},$ $H_3(x,y)=y_{n+1}^{-\frac12}.$ By Leibniz rule,
$$\partial_{x,y}^{\alpha}H=\partial_{x,y}^{\alpha}(H_1H_2H_3)=\sum_{\substack{\alpha_1,\alpha_2,\alpha_3\in\mathbb{Z}_+^{2n+2}\\ \alpha_1+\alpha_2+\alpha_3=\alpha}}c_{\alpha_1,\alpha_2,\alpha_3}\partial_{x,y}^{\alpha_1}H_1\cdot\partial_{x_{n+1}}^{\alpha_2}H_2\cdot\partial_{y_{n+1}}^{\alpha_3}H_3$$
for some constant $c_{\alpha_1,\alpha_2,\alpha_3}>0$. Clearly,
$$|(\partial_{x,y}^{\alpha_1}H_1)(x,y)|\leq c_{1,\alpha_1}|x-y|^{1-|\alpha_1|_1}$$
for some constant $c_{1,\alpha_1}>0$ and
$$(\partial_{x_{n+1}}^{\alpha_2}H_2)(x,y)=c_{2,\alpha_2}x_{n+1}^{-\frac12-|\alpha_2|_1},\quad (\partial_{y_{n+1}}^{\alpha_3}H_3)(x,y)=c_{3,\alpha_3}y_{n+1}^{-\frac12-|\alpha_3|_1}$$
for some constants $c_{2,\alpha_2}$ and $c_{3,\alpha_3}$, which are allowed to be zero.  Hence,
\begin{align*}
&(\partial_{x_{n+1}}^{\alpha_2}H_2\cdot\partial_{y_{n+1}}^{\alpha_3}H_3)(x,y)\\
&=c_{2,\alpha_2}c_{3,\alpha_3}\cdot (x_{n+1}y_{n+1})^{-\frac12-\frac12|\alpha_2|_1-\frac12|\alpha_3|_1}\cdot (\frac{x_{n+1}}{y_{n+1}})^{\frac12(|\alpha_3|_1-|\alpha_2|_1)}.
\end{align*}
Using Lemma \ref{ratio is bounded} and the assumption $H(x,y)\leq 1,$ we deduce that
\begin{align*}
&|(\partial_{x_{n+1}}^{\alpha_2}H_2\cdot\partial_{y_{n+1}}^{\alpha_3}H_3)(x,y)|\\
&\leq |c_{2,\alpha_2}c_{3,\alpha_3}|\cdot (x_{n+1}y_{n+1})^{-\frac12-\frac12|\alpha_2|_1-\frac12|\alpha_3|_1}\cdot (\frac{3+\sqrt{5}}{2})^{\frac12\big||\alpha_2|_1-|\alpha_3|_1\big|}\\
&\leq |c_{2,\alpha_2}c_{3,\alpha_3}|\cdot (\frac{3+\sqrt{5}}{2})^{\frac12\big||\alpha_2|_1-|\alpha_3|_1\big|}\cdot |x-y|^{-1-|\alpha_2|_1-|\alpha_3|_1}.
\end{align*}	
Combining these estimates, we complete the proof.
\end{proof}

\begin{lemma}\label{pre-parcet lemma} Let $\alpha\in\mathbb{Z}^{2n+2}_+,$ then there exists a constant $C_\alpha>0$ such that for every $F\in C^{|\alpha|_1}[0,\infty)$ supported on $[0,1],$ we have
$$|x-y|^{|\alpha|_1}|\partial_{x,y}^{\alpha}(F\circ H)(x,y)|\leq C_{\alpha}\|F\|_{C^{|\alpha|_1}[0,1]},\quad x,y\in\mathbb{R}^{n+1}.$$
\end{lemma}
\begin{proof}
Since $F$ is supported on $[0,1]$, it suffices to prove the assertion for those $x,y\in\mathbb{R}^{n+1}$ with $H(x,y)\leq 1.$

We prove the assertion by induction on $|\alpha|_1.$ Base of induction is obvious. It suffices to prove the step of induction. Suppose the assertion is established for all $\alpha\in\mathbb{Z}^{2n+2}_+$ with $|\alpha|_1\leq m.$
	
Let $\alpha\in\mathbb{Z}^{2n+2}_+$ be such that $|\alpha|_1=m+1.$ Choose $1\leq k\leq 2n+2$ such that $\alpha\geq e_k$ and write
$$\partial_{x,y}^{\alpha}(F\circ H)=\partial_{x,y}^{\alpha-e_k}((F^{(1)}\circ H)\cdot \partial_{x,y}^{e_k}H)=\sum_{\substack{\alpha_1,\alpha_2\in\mathbb{Z}^{2n+2}_+\\ \alpha_1+\alpha_2=\alpha-e_k}}c_{\alpha_1,\alpha_2}\partial_{x,y}^{\alpha_1}(F^{(1)}\circ H)\cdot \partial_{x,y}^{\alpha_2+e_k}H.$$
By Lemma \ref{H derivatives lemma} and by the inductive assumption, we have
\begin{align*}
|\partial_{x,y}^{\alpha}(F\circ H)(x,y)|&\leq\sum_{\substack{\alpha_1,\alpha_2\in\mathbb{Z}^{2n+2}_+\\ \alpha_1+\alpha_2=\alpha-e_k}}c_{\alpha_1,\alpha_2}\frac{c_{\alpha_1}\|F^{(1)}\|_{C^{|\alpha_1|_1}[0,1]}}{|x-y|^{|\alpha_1|_1}}\cdot \frac{c_{\alpha_2+e_k}}{|x-y|^{|\alpha_2|_1+1}}\\
&\leq |x-y|^{-|\alpha|_1}\cdot \sum_{\substack{\alpha_1,\alpha_2\in\mathbb{Z}^{2n+2}_+\\ \alpha_1+\alpha_2=\alpha-e_k}}c_{\alpha_1,\alpha_2}c_{\alpha_1}c_{\alpha_2+e_k}\cdot \|F\|_{C^{m+1}[0,1]}.
\end{align*}
Therefore, the assertion holds for $|\alpha|_1=m+1$.
\end{proof}

\begin{lemma}\label{post-parcet lemma} Let $F\in C^{1+\lfloor\frac{n+1}{2}\rfloor}[0,\infty)$ be a function supported on $[0,1].$ Then for every $1<p<\infty$, there is a constant $C_{n,p}>0$ such that
$$\|\mathfrak{S}_{F\circ H}\|_{\mathcal{L}_p(L_2(\mathbb{R}^{n+1}))\circlearrowleft}\leq C_{n,p}\|F\|_{C^{1+\lfloor\frac{n+1}{2}\rfloor}[0,1]}.$$
\end{lemma}
\begin{proof} The assertion follows by combining Lemma \ref{pre-parcet lemma} and Theorem \ref{Annalspaper}.
\end{proof}

\begin{lemma}\label{post-denis lemma} For every $1<p<\infty,$ there is a constant $C_{n,p}>0$ such that
$$\|\mathfrak{S}_{F\circ H}\|_{\mathcal{L}_p(L_2(\mathbb{R}^{n+1}))\circlearrowleft}\leq C_{n,p}\|F\circ\exp\|_{W^{1+\lceil\frac{n+1}{2}\rceil,2}(\mathbb{R})}.$$
\end{lemma}
\begin{proof} By Fourier inversion formula,
$$F(e^t)=\frac1{\sqrt{2\pi}}\int_{-\infty}^{\infty}\mathcal{F}(F\circ\exp)(s)e^{its}ds,\quad t\in\mathbb{R}.$$
In other words,
$$F(t)=\frac1{\sqrt{2\pi}}\int_{-\infty}^{\infty}\mathcal{F}(F\circ\exp)(s)t^{is}ds,\quad t>0.$$
Thus,
$$(F\circ H)(x,y)=\frac1{\sqrt{2\pi}}\int_{-\infty}^{\infty}\mathcal{F}(F\circ\exp)(s)|x-y|^{is}|x_{n+1}|^{-\frac{is}{2}}|y_{n+1}|^{-\frac{is}{2}}ds,\quad x,y\in\mathbb{R}^{n+1}.$$
Denote
$$m_s(x,y)=|x-y|^{is},\quad x,y\in\mathbb{R}^{n+1},\quad s\in\mathbb{R}.$$
We have
$$\mathfrak{S}_{F\circ H}=\frac1{\sqrt{2\pi}}\int_{-\infty}^{\infty}\mathcal{F}(F\circ\exp)(s)M_{|x_{n+1}|^{-\frac{is}{2}}}\mathfrak{S}_{m_s}M_{|x_{n+1}|^{-\frac{is}{2}}}ds.$$
A standard transference technique (see e.g.  \cite{MR3378821,MR2866074}) yields
\begin{align*}
\|\mathfrak{S}_{m_s}\|_{\mathcal{L}_p(L_2(\mathbb{R}^{n+1}))\circlearrowleft}\leq \|{\rm id}\otimes \Delta^{is}\|_{L_p(B(L_2(\mathbb{R}^{n+1}))\bar{\otimes}L_{\infty}(\mathbb{R}^{n+1}))\circlearrowleft},\quad s\in\mathbb{R}.
\end{align*}
Here we used the notation $\bar{\otimes}$ to denote the von Neumann algebra tensor product and $L_p(B(L_2(\mathbb{R}^{n+1}))\bar{\otimes}L_{\infty}(\mathbb{R}^{n+1}))$ to denote the non-commutative $L_p$ space over von Neumann algebra $B(L_2(\mathbb{R}^{n+1}))\bar{\otimes}L_{\infty}(\mathbb{R}^{n+1})$ (see e.g. \cite{MR4738490,PX} for more details about non-commutative integration theory). Taking into account that the space $\mathcal{L}_p(L_2(\mathbb{R}^{n+1}))$ has UMD property (see \cite[Definition 4.14]{HNVWbook}) and using \cite[Corollary 5.70]{HNVWbook}, we conclude that
\begin{align*}
\|{\rm id}\otimes \Delta^{is}\|_{L_p(B(L_2(\mathbb{R}^{n+1}))\bar{\otimes}L_{\infty}(\mathbb{R}^{n+1}))\circlearrowleft}\leq C_{n,p}(1+|s|)^{\frac{n+1}{2}},\quad s\in\mathbb{R}.
\end{align*}
Thus,
$$\|\mathfrak{S}_{m_s}\|_{\mathcal{L}_p(L_2(\mathbb{R}^{n+1}))\circlearrowleft}\leq C_{n,p}(1+|s|)^{\frac{n+1}{2}},\quad s\in\mathbb{R}.$$
Combining with the fact that $M_{|x_{n+1}|^{-\frac{is}{2}}}$ is a unitary operator on $L_2(\mathbb{R}^{n+1})$, we deduce that
\begin{align*}
&\|\mathfrak{S}_{F\circ H}\|_{\mathcal{L}_p(L_2(\mathbb{R}^{n+1}))\circlearrowleft}\\
&\leq\frac1{\sqrt{2\pi}}\int_{-\infty}^{\infty}|\mathcal{F}(F\circ\exp)(s)|\big\|M_{|x_{n+1}|^{-\frac{is}{2}}}\mathfrak{S}_{m_s}M_{|x_{n+1}|^{-\frac{is}{2}}}\big\|_{\mathcal{L}_p(L_2(\mathbb{R}^{n+1}))\circlearrowleft}ds\\
&\leq C_{n,p}\int_{-\infty}^{\infty}|\mathcal{F}(F\circ\exp)(s)| (1+|s|)^{\frac{n+1}{2}}ds\\
&\leq C_{n,p}\|F\circ\exp\|_{W^{1+\lceil\frac{n+1}{2}\rceil,2}(\mathbb{R})}.
\end{align*}
Here, the last inequality follows from \cite[Lemma 7]{PS-crelle}. This completes the proof of Lemma \ref{post-denis lemma}.
\end{proof}

\begin{proof}[Proof of Theorem \ref{main schur theorem}] Let $\phi\in C^{\infty}[0,\infty)$ be supported on $[0,1]$ and such that $\phi(0)=1.$ It follows from Lemma \ref{post-parcet lemma} that $\mathfrak{S}_{\phi\circ H}$ is a bounded mapping from $\mathcal{L}_p(L_2(\mathbb{R}^{n+1}))$ to itself for every $1<p<\infty.$ By considering $F-F(0)\cdot\phi$ instead of $F,$ we may assume without loss of generality that $F(0)=0.$ By assumption, $F\circ\exp\in W^{1+\lceil\frac{n+1}{2}\rceil,2}(\mathbb{R}_+)$ and $F\circ\exp\in W^{1+\lceil\frac{n+1}{2}\rceil,2}(\mathbb{R}_-).$ Hence, $F\circ\exp\in W^{1+\lceil\frac{n+1}{2}\rceil,2}(\mathbb{R})$ and the assertion follows from Lemma \ref{post-denis lemma}.
\end{proof}

\begin{proposition}\label{fkl schur lemma} If $(k,l)\in\{(2,0),(1,1),(2,1)\}$ and $1<p<\infty$, then the Schur multipliers $\mathfrak{S}_{F_{k,l}\circ H}$ are bounded from $\mathcal{L}_p(L_2(\mathbb{R}^{n+1}_+))$ to itself. Consequently, those Schur multipliers are bounded from $\mathcal{L}_{p,\infty}(L_2(\mathbb{R}^{n+1}_+))$ to itself.
\end{proposition}
\begin{proof} By Theorem \ref{fkl are smooth theorem}, $F_{k,l}$ satisfies the conditions in Theorem \ref{main schur theorem}. This, in combination with Theorem \ref{main schur theorem}, yields that $\mathfrak{S}_{F_{k,l}\circ H}$ is bounded on $\mathcal{L}_p(L_2(\mathbb{R}^{n+1}_+)),$ $1<p<\infty.$ The second statement follows by interpolation.
\end{proof}

\begin{proposition}\label{standard schur lemma}
If $1\leq k\leq n+1$ and $1<p<\infty$, then Schur multipliers $\mathfrak{S}_a,$ $\mathfrak{S}_b$ and $\mathfrak{S}_{h_k}$ are bounded from $\mathcal{L}_p(L_2(\mathbb{R}^{n+1}_+))$ to itself. Consequently, those Schur multipliers are bounded from $\mathcal{L}_{p,\infty}(L_2(\mathbb{R}^{n+1}_+))$ to itself and from $(\mathcal{L}_{p,\infty}(L_2(\mathbb{R}^{n+1}_+)))_0$ to itself.
\end{proposition}
\begin{proof} Let us start with $\mathfrak{S}_a.$ We have
$$e^{-|t|}=\frac1{\pi}\int_{\mathbb{R}}\frac{e^{its}ds}{1+s^2},\quad t\in\mathbb{R}.$$
Setting $t=\frac{1}{2}\log(\frac{x_{n+1}}{y_{n+1}}),$ we write
$$a(x,y)=\frac1{\pi}\int_{\mathbb{R}}x_{n+1}^{\frac{is}{2}}y_{n+1}^{-\frac{is}{2}}\frac{ds}{1+s^2},\quad t\in\mathbb{R}.$$
Thus,
$$\mathfrak{S}_a(V)=\frac1{\pi}\int_{\mathbb{R}}M_{x_{n+1}^{\frac{is}{2}}}VM_{x_{n+1}^{-\frac{is}{2}}}\frac{ds}{1+s^2}.$$
Consequently,
\begin{align*}
\|\mathfrak{S}_a(V)\|_{\mathcal{L}_p(L_2(\mathbb{R}^{n+1}_+))}&\leq\frac1{\pi}\int_{\mathbb{R}}\|M_{x_{n+1}^{\frac{is}{2}}}VM_{x_{n+1}^{-\frac{is}{2}}}\|_{\mathcal{L}_p(L_2(\mathbb{R}^{n+1}_+))}\frac{ds}{1+s^2}\\
&=\frac1{\pi}\int_{\mathbb{R}}\|V\|_{\mathcal{L}_p(L_2(\mathbb{R}^{n+1}_+))}\frac{ds}{1+s^2}=\|V\|_{\mathcal{L}_p(L_2(\mathbb{R}^{n+1}_+))}.
\end{align*}
Next, $\mathfrak{S}_b$ is the triangular truncation operator on the semifinite von Neumann algebra $\mathcal{M}=B(L_2(\mathbb{R}^{n+1}_+))$ with respect to the spectral measure of the operator $M_{x_{n+1}}.$ Boundedness of triangular truncation operator on
$\mathcal{L}_p(L_2(\mathbb{R}^{n+1}_+))=L_p(\mathcal{M})$ follows from Macaev theorem. We refer the reader to \cite{MSZ-krein} for a detailed proof of Macaev theorem.

Fourier multiplier $m_k(\nabla)$ associated with $m_k(x):=\frac{x_{k}}{|x|}$ is completely bounded on $L_p(\mathbb{R}^{n+1}).$ By the standard transference technique (see e.g.  \cite{MR3378821,MR2866074}), we conclude that
\begin{align*}
\|\mathfrak{S}_{h_k}\|_{\mathcal{L}_p(L_2(\mathbb{R}_+^{n+1}))\circlearrowleft}&\leq \|\mathfrak{S}_{h_k}\|_{\mathcal{L}_p(L_2(\mathbb{R}^{n+1}))\circlearrowleft}\\
&\leq\|{\rm id}\otimes m_k(\nabla)\|_{L_p(B(L_2(\mathbb{R}^{n+1}))\bar{\otimes}L_{\infty}(\mathbb{R}^{n+1}))\circlearrowleft}<\infty.
\end{align*}
The second statement follows by interpolation.
\end{proof}

\section{Proof of the upper bound}\label{upper bound section}
\setcounter{equation}{0}

This section is devoted to providing a proof of upper bound in Theorem \ref{main theorem}.

\begin{lemma}\label{mta final lemma} There exists a constant $C_{n,\lambda}>0$ such that for every $f\in C^{\infty}_c(\mathbb{R}^{n+1})$ and for every $1\leq k\leq n+1,$
$$\|[R_{\lambda,k},M_{E^{\ast}f}]\|_{\mathcal{L}_{n+1,\infty}(L_2(\mathbb{R}^{n+1}_+,m_{\lambda}))}\leq C_{n,\lambda}\|f\|_{\dot{W}^{1,n+1}(\mathbb{R}^{n+1})}.$$
\end{lemma}
\begin{proof} Proposition \ref{commutator representation lemma} yields
\begin{align*}
&\|[R_{\lambda,k},M_{E^{\ast}f}]\|_{\mathcal{L}_{n+1,\infty}(L_2(\mathbb{R}^{n+1}_+,m_{\lambda}))}\\
&\leq \kappa_{n,\lambda}^{[3]}\Big\|\mathfrak{S}_{F_{2,0}\circ H}\Big\|_{\mathcal{L}_{n+1,\infty}(L_2(\mathbb{R}^{n+1}_+,m_{\lambda}))\circlearrowleft}\\
&\hspace{1.0cm}\times\Big\| M_{x_{n+1}^{-\lambda}} E^{\ast}[R_k,M_f]E M_{x_{n+1}^{\lambda}}\Big\|_{\mathcal{L}_{n+1,\infty}(L_2(\mathbb{R}^{n+1}_+,m_{\lambda}))}\\
&+\kappa_{n,\lambda}^{[3]}\sum_{l=1}^{n+1}\Big\|\mathfrak{S}_{h_l}\circ\mathfrak{S}_a\circ\mathfrak{S}_{F_{1,1}\circ H}\Big\|_{\mathcal{L}_{n+1,\infty}(L_2(\mathbb{R}^{n+1}_+,m_{\lambda}))\circlearrowleft}\\
&\hspace{1.0cm}\times\Big\|M_{x_{n+1}^{-\lambda}} E^{\ast}[R_l,M_f]E M_{x_{n+1}^{\lambda}}\Big\|_{\mathcal{L}_{n+1,\infty}(L_2(\mathbb{R}^{n+1}_+,m_{\lambda}))}\\
&+\kappa_{n,\lambda}^{[3]}\sum_{l=1}^{n+1}\Big\|\mathfrak{S}_{h_l}\circ\mathfrak{S}_{h_{n+1}}\circ\mathfrak{S}_b\circ\mathfrak{S}_{F_{2,1}\circ H}\Big\|_{\mathcal{L}_{n+1,\infty}(L_2(\mathbb{R}^{n+1}_+,m_{\lambda}))\circlearrowleft}\\
&\hspace{1.0cm}\times\Big\|M_{x_{n+1}^{-\lambda}} E^{\ast}[R_l,M_f]E M_{x_{n+1}^{\lambda}}\Big\|_{\mathcal{L}_{n+1,\infty}(L_2(\mathbb{R}^{n+1}_+,m_{\lambda}))}.
\end{align*}
It follows from Lemma \ref{weight} and Propositions \ref{fkl schur lemma} and \ref{standard schur lemma} that there exists a constant $C_{n,\lambda}>0$ such that
\begin{align*}
\Big\|\mathfrak{S}_{F_{2,0}\circ H}\Big\|_{\mathcal{L}_{n+1,\infty}(L_2(\mathbb{R}^{n+1}_+,m_{\lambda}))\circlearrowleft}=\Big\|\mathfrak{S}_{F_{2,0}\circ H}\Big\|_{\mathcal{L}_{n+1,\infty}(L_2(\mathbb{R}^{n+1}_+))\circlearrowleft}\stackrel{P.\ref{fkl schur lemma}}{\leq}{C_{n,\lambda}},
\end{align*}
and that
\begin{align*}
&\Big\|\mathfrak{S}_{h_l}\circ\mathfrak{S}_a\circ\mathfrak{S}_{F_{1,1}\circ H}\Big\|_{\mathcal{L}_{n+1,\infty}(L_2(\mathbb{R}^{n+1}_+,m_{\lambda}))\circlearrowleft}\\
&=\Big\|\mathfrak{S}_{h_l}\circ\mathfrak{S}_a\circ\mathfrak{S}_{F_{1,1}\circ H}\Big\|_{\mathcal{L}_{n+1,\infty}(L_2(\mathbb{R}^{n+1}_+))\circlearrowleft}\\
&\leq \Big\|\mathfrak{S}_{h_l}\Big\|_{\mathcal{L}_{n+1,\infty}(L_2(\mathbb{R}^{n+1}_+))\circlearrowleft} \Big\|\mathfrak{S}_a\Big\|_{\mathcal{L}_{n+1,\infty}(L_2(\mathbb{R}^{n+1}_+))\circlearrowleft} \Big\|\mathfrak{S}_{F_{1,1}\circ H}\Big\|_{\mathcal{L}_{n+1,\infty}(L_2(\mathbb{R}^{n+1}_+))\circlearrowleft}\\
&\stackrel{P.\ref{fkl schur lemma},\ref{standard schur lemma}}{\leq}{C_{n,\lambda}},
\end{align*}
and that
\begin{align*}
&\Big\|\mathfrak{S}_{h_l}\circ\mathfrak{S}_{h_{n+1}}\circ\mathfrak{S}_b\circ\mathfrak{S}_{F_{2,1}\circ H}\Big\|_{\mathcal{L}_{n+1,\infty}(L_2(\mathbb{R}^{n+1}_+,m_{\lambda}))\circlearrowleft}\\
&=\Big\|\mathfrak{S}_{h_l}\circ\mathfrak{S}_{h_{n+1}}\circ\mathfrak{S}_b\circ\mathfrak{S}_{F_{2,1}\circ H}\Big\|_{\mathcal{L}_{n+1,\infty}(L_2(\mathbb{R}^{n+1}_+))\circlearrowleft}\\
&\leq \Big\|\mathfrak{S}_{h_l}\Big\|_{\mathcal{L}_{n+1,\infty}(L_2(\mathbb{R}^{n+1}_+))\circlearrowleft} \Big\|\mathfrak{S}_{h_{n+1}}\Big\|_{\mathcal{L}_{n+1,\infty}(L_2(\mathbb{R}^{n+1}_+))\circlearrowleft}\\
&\times\Big\|\mathfrak{S}_b\Big\|_{\mathcal{L}_{n+1,\infty}(L_2(\mathbb{R}^{n+1}_+))\circlearrowleft}\times \Big\|\mathfrak{S}_{F_{2,1}\circ H}\Big\|_{\mathcal{L}_{n+1,\infty}(L_2(\mathbb{R}^{n+1}_+))\circlearrowleft}\\
&\stackrel{P.\ref{fkl schur lemma},\ref{standard schur lemma}}{\leq}{C_{n,\lambda}}.
\end{align*}
This, in combination with Lemma \ref{half}, yields
\begin{align*}
&\|[R_{\lambda,k},M_{E^{\ast}f}]\|_{\mathcal{L}_{n+1,\infty}(L_2(\mathbb{R}^{n+1}_+,m_{\lambda}))}\\
&\leq C_{n,\lambda}\sum_{l=1}^{n+1}\Big\|M_{x_{n+1}^{-\lambda}} E^{\ast}[R_l,M_f]E M_{x_{n+1}^{\lambda}}\Big\|_{\mathcal{L}_{n+1,\infty}(L_2(\mathbb{R}^{n+1}_+,m_{\lambda}))}\\
&=C_{n,\lambda}\sum_{l=1}^{n+1}\Big\|E^{\ast}[R_l,M_f]E\Big\|_{\mathcal{L}_{n+1,\infty}(L_2(\mathbb{R}^{n+1}_+))}\\
&=C_{n,\lambda} \sum_{l=1}^{n+1}\Big\|M_{\chi_{\mathbb{R}_+^{n+1}}}[R_l,M_f]M_{\chi_{\mathbb{R}_+^{n+1}}}\Big\|_{\mathcal{L}_{n+1,\infty}(L_2(\mathbb{R}^{n+1}))}\\
&\leq C_{n,\lambda} \sum_{l=1}^{n+1}\Big\|[R_l,M_f]\Big\|_{\mathcal{L}_{n+1,\infty}(L_2(\mathbb{R}^{n+1}))}.
\end{align*}
The assertion follows now from the corresponding Euclidean assertion (see \cite[Theorem 1]{LMSZ}).
\end{proof}

\begin{lemma}\label{density lemma} Let $1<p<\infty.$ For every $f \in \dot{W}^{1,p}(\mathbb{R}_+^{n+1})\cap L_{\infty}(\mathbb{R}_+^{n+1}),$ there exists a sequence $\{f_m\}_{m=1}^\infty \subset C^{\infty}_c(\mathbb{R}^{n+1})$ and a constant $c>0$ such that
\begin{enumerate}[{\rm (i)}]
\item $f_m$ is even in the $(n+1)$-st argument;
\item $\{f_m\}_{m=1}^\infty$ is a Cauchy sequence in $\dot{W}^{1,p}(\mathbb{R}^{n+1});$
\item $E^{\ast}f_m\to f$ in $\dot{W}^{1,p}(\mathbb{R}_+^{n+1});$
\item $M_{E^{\ast}f_m}\rightarrow M_{f-c}$ in the strong operator topology on $B(L_2(\mathbb{R}^{n+1}_+))$.
\end{enumerate}	
\end{lemma}
\begin{proof} Set
$$F(x):=
\begin{cases}
f(x_1,\cdots,x_n,x_{n+1}),& x_{n+1}>0\\
f(x_1,\cdots,x_n,-x_{n+1}),& x_{n+1}<0
\end{cases},\quad x=(x_1,\cdots,x_{n+1})\in\mathbb{R}^{n+1}.
$$
It is immediate that $F\in\dot{W}^{1,p}(\mathbb{R}^{n+1})\cap L_{\infty}(\mathbb{R}^{n+1}).$ The assertion follows now from the corresponding assertion for $\mathbb{R}^{n+1}$ (see \cite[Theorem 3]{LMSZ}).
\end{proof}

\begin{proof}[Proof of upper bound in Theorem \ref{main theorem}] If $f=E^{\ast}g$ for some $g\in C_c^\infty(\mathbb{R}^{n+1})$ which is even in the $(n+1)$-st argument.
	
We apply an approximation argument to remove the above assumption. To this end, we suppose $f\in  \dot{W}^{1,n+1}(\mathbb{R}_+^{n+1})\cap L_{\infty}(\mathbb{R}_+^{n+1})$ and let $\{f_m\}_{m\geq 1}$ be the sequence given by Lemma \ref{density lemma}, so that $f_m\in C_c^\infty(\mathbb{R}^{n+1})$ for $m\geq 1$ and $\{f_m\}_{m\geq 1}$ is a Cauchy sequence on $\dot{W}^{1,n+1}(\mathbb{R}^{n+1})$. Using Lemma \ref{mta final lemma}, we deduce that for $m_1,m_2\geq1$,
\begin{align*}
&\|[R_{\lambda,k},M_{E^{\ast}f_{m_1}}]-[R_{\lambda,k},M_{E^{\ast}f_{m_2}}]\|_{\mathcal{L}_{n+1,\infty}(L_2(\mathbb{R}^{n+1}_+,m_{\lambda}))}\\
&=\|[R_{\lambda,k},M_{E^{\ast}(f_{m_1}-f_{m_2})}]\|_{\mathcal{L}_{n+1,\infty}(L_2(\mathbb{R}^{n+1}_+,m_{\lambda}))}\leq C_{n,\lambda}\|f_{m_1}-f_{m_2}\|_{\dot{W}^{1,n+1}(\mathbb{R}^{n+1})}.
\end{align*}
Hence, $\{[R_{\lambda,k},M_{E^{\ast}f_{m}}]\}_{m\geq 1}$ is a Cauchy sequence on $\mathcal{L}_{n+1,\infty}(L_2(\mathbb{R}^{n+1}_+,m_{\lambda}))$ which, therefore, converges to some $A\in \mathcal{L}_{n+1,\infty}(L_2(\mathbb{R}^{n+1}_+,m_{\lambda})).$ In particular, $$[R_{\lambda,k},M_{E^{\ast}f_{m}}]\rightarrow A$$
in the strong operator topology. On the other hand, by Lemma \ref{density lemma}, $M_{E^{\ast}f_m}\rightarrow M_{f-c}$ for some constant $c>0$ in the strong operator topology. Therefore, $$[R_{\lambda,k},M_{E^{\ast}f_m}]\rightarrow [R_{\lambda,k},M_f]$$
in the strong operator topology. By uniqueness of the limit, $A=[R_{\lambda,k},M_f]$, which implies that  $$[R_{\lambda,k},M_{E^{\ast}f_m}]\rightarrow [R_{\lambda,k},M_f]$$
in $\mathcal{L}_{n+1,\infty}(L_2(\mathbb{R}^{n+1}_+,m_{\lambda})).$ Thus,
\begin{align*}
\|[R_{\lambda,k},M_f]\|_{\mathcal{L}_{n+1,\infty}(L_2(\mathbb{R}^{n+1}_+,m_{\lambda}))}&=\lim_{m\rightarrow \infty}\|[R_{\lambda,k},M_{E^{\ast}f_m}]\|_{\mathcal{L}_{n+1,\infty}(L_2(\mathbb{R}^{n+1}_+,m_{\lambda}))}\\
&\leq C_{n,\lambda}\limsup_{m\rightarrow \infty}\|f_m\|_{\dot{W}^{1,n+1}(\mathbb{R}_+^{n+1})}\\
&=C_{n,\lambda}\|f\|_{\dot{W}^{1,n+1}(\mathbb{R}_+^{n+1})}.
\end{align*}
This completes the proof.
\end{proof}

\section{Alternative proof of upper bound for $n\geq 2$}\label{alternativepf}
\setcounter{equation}{0}

This section can be regarded as a separate section, in which we provide an alternative proof of upper bound in Theorem \ref{main theorem} via establishing corresponding Cwikel estimates in the Bessel setting.

\begin{lemma}\label{L2Cwikel}
There exists a constant $C_{n,\lambda}>0$ such that for every $f\in L_2(\mathbb{R}_+^{n+1})$ and for every $g\in L_2(\mathbb{R}_+,r^ndr)$, we have
$$\|M_fg(\sqrt{\Delta_{\lambda}})\|_{\mathcal{L}_2(L_2(\mathbb{R}_+^{n+1}))}\leq C_{n,\lambda}\|f\|_{L_2(\mathbb{R}_+^{n+1})}\|g\|_{L_2(\mathbb{R}_+,r^ndr)}.$$
\end{lemma}
\begin{proof} To begin with, by functional calculus \cite{MR617913},
$$g(\sqrt{\Delta_{\lambda}})=(\mathcal{F}\otimes U_{\lambda})^{-1}M_{g(|x|)}(\mathcal{F}\otimes U_{\lambda}).$$
Therefore,
\begin{align*}
g(\sqrt{\Delta_{\lambda}})(f)(x)&=\frac{1}{(2\pi)^n}\int_{\mathbb{R}_+^{n+1}}\int_{\mathbb{R}_+^{n+1}}e^{i\langle y',x'-z'\rangle}\times\\
&\hspace{1.0cm}\times\phi_\lambda(x_{n+1}y_{n+1})g(|y|)\phi_\lambda(y_{n+1}z_{n+1})dm_\lambda(y)f(z)dm_\lambda(z).
\end{align*}
Hence, the integral kernel of $g(\sqrt{\Delta_{\lambda}})$ is of the following form
$$K_{g(\sqrt{\Delta_{\lambda}})}(x,\,y)=\frac{1}{(2\pi)^n}\int_{\mathbb{R}_+^{n+1}}e^{i\langle z',x'-y'\rangle}\phi_\lambda(x_{n+1}z_{n+1})g(|z|)\phi_\lambda(z_{n+1}y_{n+1})dm_\lambda(z).$$
Integral kernel on the diagonal is
$$\frac{1}{(2\pi)^n}\int_{\mathbb{R}_+^{n+1}}\phi_\lambda^2(x_{n+1}z_{n+1})g(|z|)dm_\lambda(z).$$
Now we calculate the $\mathcal{L}_2$ norm of $M_fg(\sqrt{\Delta_{\lambda}})$ as follows.
\begin{align*}
&\|M_fg(\sqrt{\Delta_{\lambda}})\|_{\mathcal{L}_2(L_2(\mathbb{R}_+^{n+1}))}^2\\&={\rm Trace}(M_{|f|^2}|g|^2(\sqrt{\Delta_{\lambda}}))\\
&=\frac{1}{(2\pi)^n}\int_{\mathbb{R}_+^{n+1}}|f(x)|^2\Big(\int_{\mathbb{R}_+^{n+1}}\phi_\lambda^2(x_{n+1}z_{n+1})|g(|z|)|^2dm_\lambda(z)\Big)dm_\lambda(x)\\
&=\frac{1}{(2\pi)^n}\int_{\mathbb{R}_+^{n+1}}|f(x)|^2\Big(\int_{\mathbb{R}_+^{n+1}}\psi_{\lambda}^2(x_{n+1}z_{n+1})|g(|z|)|^2dz\Big)dx.
\end{align*}
Here, $\psi_{\lambda}:t\to t^{\lambda}\phi_{\lambda}(t)=t^{\frac12}J_{\lambda-\frac12}(t),$ $t>0.$ Recall from \cite[p.364]{MR167642} that the function $\psi_{\lambda}$ is bounded on $(0,\infty).$ Thus,
\begin{align*}
\|M_fg(\sqrt{\Delta_{\lambda}})\|_{\mathcal{L}_2(L_2(\mathbb{R}_+^{n+1}))}^2
&\leq\frac{1}{(2\pi)^n}\|\psi_{\lambda}\|_{\infty}^2\int_{\mathbb{R}_+^{n+1}}|f(x)|^2\Big(\int_{\mathbb{R}_+^{n+1}}|g(|z|)|^2dz\Big)dx\\
&=\frac{1}{(2\pi)^n}\|\psi_{\lambda}\|_{\infty}^2\int_{\mathbb{R}_+^{n+1}}|f(x)|^2dx\cdot \int_{\mathbb{R}_+^{n+1}}|g(|z|)|^2dz\\
&=\frac{\upsilon_n}{(2\pi)^n}\|\psi_{\lambda}\|_{\infty}^2\int_{\mathbb{R}_+^{n+1}}|f(x)|^2dx\cdot \int_{\mathbb{R}_+}|g(r)|^2r^ndr,
\end{align*}
where $\upsilon_n$ denotes the volumn of $\mathbb{S}^n_+$. This ends the proof of Lemma \ref{L2Cwikel}.
\end{proof}

Combining  Lemma \ref{L2Cwikel} with the Abstract Cwikel Estimate established in \cite{LeSZ} (see also \cite[Theorem 1.5.10]{LSZ2}), we obtain the Cwikel estimate for $M_fg(\sqrt{\Delta_{\lambda}})$ in the case of $p>2$ and $n\geq 2.$

\begin{theorem}\label{cwikel estimate in bessel setting}
Let $p>2$ and $n\geq 2,$ then there exists a constant $C_{n,p,\lambda}>0$ such that
\begin{enumerate}[{\rm (i)}]
\item for every $f\in L_p(\mathbb{R}_+^{n+1}),$ $g\in L_p(\mathbb{R}_+,r^ndr),$  we have
$$\|M_fg(\sqrt{\Delta_{\lambda}})\|_{\mathcal{L}_p(L_2(\mathbb{R}_+^{n+1}))}\leq C_{n,p,\lambda}\|f\|_{L_p(\mathbb{R}_+^{n+1})}\|g\|_{L_{p}(\mathbb{R}_+,r^ndr)},$$
\item for every $f\in L_p(\mathbb{R}_+^{n+1}),$ $g\in L_{p,\infty}(\mathbb{R}_+,r^ndr),$ we have
$$\|M_fg(\sqrt{\Delta_{\lambda}})\|_{\mathcal{L}_{p,\infty}(L_2(\mathbb{R}_+^{n+1}))}\leq C_{n,p,\lambda}\|f\|_{L_p(\mathbb{R}_+^{n+1})}\|g\|_{L_{p,\infty}(\mathbb{R}_+,r^ndr)}.$$
\end{enumerate}
\end{theorem}

\begin{lemma}\cite[Lemma 4.4]{MR4654013}\label{surprisingly_technical_lemma}
Let $B\geq 0$ be a (potentially unbounded) linear operator on a Hilbert space $H$ with $ker(B)=0$, and let $A$ be a bounded operator on $H$. Suppose that $1<p<\infty$ and
\begin{enumerate}[{\rm (i)}]
\item $B^{-1}[B^2,A]B^{-1}\in\mathcal{L}_{p,\infty}(H);$
\item $[B,A]B^{-1}\in\mathcal{L}_{\infty}(H);$
\item $AB^{-1}\in\mathcal{L}_{p,\infty}(H)$ or $B^{-1}A\in\mathcal{L}_{p,\infty}(H).$
\end{enumerate}
Under these assumptions, there exists a constant $C_p>0$ such that
$$\|[B,A]B^{-1}\|_{\mathcal{L}_{p,\infty}(H)} \leq C_p\|B^{-1}[B^2,A]B^{-1}\|_{\mathcal{L}_{p,\infty}(H)}.$$
\end{lemma}

\begin{lemma}\label{sufficiency verification lemma} Let $n\geq 2$ and $p=n+1,$ then the operators $A:=M_{E^*f},$ $f\in C^{\infty}_c(\mathbb{R}^{n+1})$, and $B=\Delta_{\lambda}^{\frac12}$ satisfy the conditions in Lemma \ref{surprisingly_technical_lemma}. Furthermore, there exists a constant $C_{n,\lambda}>0$ such that
$$\|B^{-1}[B^2,A]B^{-1}\|_{\mathcal{L}_{n+1,\infty}(L_2(\mathbb{R}_+^{n+1}))}\leq C_{n,\lambda}\|f\|_{\dot{W}^{1,n+1}(\mathbb{R}^{n+1}_+)}.$$
\end{lemma}
\begin{proof}
By Leibniz rule,
\begin{align*}
[\Delta_{\lambda},M_{E^*f}]=\sum_{k=1}^{n+1}[\partial_k^{\ast}\partial_k,M_{E^*f}]&=\sum_{k=1}^{n+1}\big([\partial_k^{\ast},M_{E^*f}]\partial_k+\partial_k^{\ast}[\partial_k,M_{E^*f}]\big)\\
&=\sum_{k=1}^{n+1}\big(\partial_k^{\ast}M_{\partial_k{E^*f}}-M_{\partial_k{E^*f}}\partial_k\big).
\end{align*}
Thus,
\begin{align*}
B^{-1}[B^2,A]B^{-1}&=\Delta_\lambda^{-\frac12}[\Delta_{\lambda},M_{E^*f}]\Delta_\lambda^{-\frac12}\\
&=\sum_{k=1}^{n+1}\big(R_{\lambda,k}^{\ast}\cdot M_{\partial_k{E^*f}}\Delta_{\lambda}^{-\frac12}-\Delta_{\lambda}^{-\frac12}M_{\partial_k{E^*f}}\cdot R_{\lambda,k}\big).
\end{align*}
By Theorem \ref{cwikel estimate in bessel setting},
\begin{align*}
&\|B^{-1}[B^2,A]B^{-1}\|_{\mathcal{L}_{n+1,\infty}(L_2(\mathbb{R}_+^{n+1}))}\\
&\lesssim\sum_{k=1}^{n+1}\big(\|M_{\partial_k{E^*f}}\Delta_{\lambda}^{-\frac12}\|_{\mathcal{L}_{n+1,\infty}(L_2(\mathbb{R}_+^{n+1}))}+\|\Delta_{\lambda}^{-\frac12}M_{\partial_k{E^*f}}\|_{\mathcal{L}_{n+1,\infty}(L_2(\mathbb{R}_+^{n+1}))}\big)\\
\\&\lesssim\|f\|_{\dot{W}^{1,n+1}(\mathbb{R}^{n+1}_+)}.
\end{align*}
This verifies the first condition in Lemma \ref{surprisingly_technical_lemma}.
	
Next, we verify the second condition in Lemma \ref{surprisingly_technical_lemma}. Indeed, it follows from \cite[Proposition 7.6]{DGKLWY} that the commutators $[R_{\lambda,k},M_{E^{\ast}f}],$ $1\leq k\leq n+1,$ are bounded. Using Leibniz rule, we decompose $[R_{\lambda,k},M_{E^{\ast}f}]$ as follows.
$$[R_{\lambda,k},M_{E^{\ast}f}]=[\partial_k,M_{E^{\ast}f}]\Delta_\lambda^{-\frac12}+\partial_k[\Delta_\lambda^{-\frac12},M_{E^{\ast}f}]=M_{E^{\ast}(\partial_kf)}\Delta_\lambda^{-\frac12}-R_{\lambda,k}[B,A]B^{-1}.$$
By Theorem \ref{cwikel estimate in bessel setting}, the first summand on the right is bounded. Hence, $R_{\lambda,k}[B,A]B^{-1}$ is bounded for every $1\leq k\leq n+1.$ Taking into account that
$$|[B,A]B^{-1}|^2=\sum_{k=1}^{n+1}|R_{\lambda,k}[B,A]B^{-1}|^2,$$
it follows that $[B,A]B^{-1}$ is bounded.
	
The third condition in Lemma \ref{surprisingly_technical_lemma} follows from Theorem \ref{cwikel estimate in bessel setting}.
\end{proof}

\begin{lemma}\label{main sufficiency lemma} Let $n\geq 2,$ then there exists a constant $C_{n,\lambda}>0$ such that for every $f\in C^{\infty}_c(\mathbb{R}^{n+1}),$ we have
$$\|[\Delta_\lambda^{\frac12},M_{E^*f}]\Delta_\lambda^{-\frac12}\|_{\mathcal{L}_{n+1,\infty}(L_2(\mathbb{R}_+^{n+1}))}\leq C_{n,\lambda}\|f\|_{\dot{W}^{1,n+1}(\mathbb{R}_+^{n+1})}.$$
\end{lemma}
\begin{proof} The assertion follows from Lemma \ref{sufficiency verification lemma} and Lemma \ref{surprisingly_technical_lemma}.
\end{proof}

\begin{proposition}\label{final sufficiency lemma} Let $n\geq 2,$ then there exists a constant $C_{n,\lambda}>0$ such that for every $f\in  \dot{W}^{1,n+1}(\mathbb{R}_+^{n+1})\cap L_{\infty}(\mathbb{R}_+^{n+1})$ and for every $1\leq k\leq n+1,$ we have
$$\|[R_{\lambda,k},M_f]\|_{\mathcal{L}_{n+1,\infty}(L_2(\mathbb{R}_+^{n+1}))}\leq C_{n,\lambda}\|f\|_{\dot{W}^{1,n+1}(\mathbb{R}_+^{n+1})}.$$
\end{proposition}
\begin{proof}
By the same approximation argument given in Section \ref{upper bound section}, it suffices to show the assertion for those $f$ such that $f=E^{\ast}g$ for some $g\in C_c^\infty(\mathbb{R}^{n+1})$ which is even in the $(n+1)$-st argument. Using Leibniz rule, we decompose $[R_{\lambda,k},M_f]$ as follows.
$$[R_{\lambda,k},M_f]=[\partial_k,M_f]\Delta_\lambda^{-\frac12}+\partial_k[\Delta_\lambda^{-\frac12},M_f]=M_{\partial_k f}\Delta_\lambda^{-\frac12}-R_{\lambda,k}[\Delta_\lambda^{\frac12},M_f]\Delta_\lambda^{-\frac12},$$
where in the last step we applied the following commutator formula:
\begin{align}\label{commutator}
[B^{-1},A]=-B^{-1}[B,A]B^{-1}.
\end{align}
The assertion follows now from Lemma \ref{main sufficiency lemma}, Theorem \ref{cwikel estimate in bessel setting} and the quasi-triangle inequality.
\end{proof}

\section{Proof of the lower bound}\label{lowbd}
\setcounter{equation}{0}
This section is devoted to providing a proof of lower bound in Theorem \ref{main theorem}.

\begin{lemma}\cite[Theorem 2.13]{MR2154153}\label{positive schur lemma} Let $K_1,K_2$ be measurable functions on $B(0,R)\times B(0,R)$ and let $V_{K_1},V_{K_2}$ be integral operators with integral kernels $K_1$ and $K_2.$ If $|K_2|\leq K_1$ and if $p\in2\mathbb{N},$ then
$$\|V_{K_2}\|_{\mathcal{L}_p(L_2(B(0,R)))}\leq\|V_{K_1}\|_{\mathcal{L}_p(L_2(B(0,R)))}.$$
\end{lemma}

\begin{lemma}\label{convolution operator lemma} Let $0<\alpha<1.$ If
$$K(x,y)=|x-y|^{-\alpha(n+1)},\quad x,y\in B(0,R),$$
and if $V_K$ is an integral operator with an integral kernel $K,$ then $$V_K\in\mathcal{L}_{\frac{1}{1-\alpha},\infty}(L_2(B(0,R))).$$
\end{lemma}
\begin{proof} Conjugating with dilation operator, we can always assume $R<(n+1)^{-\frac12}$ so that $B(0,R)\subset [-1,1]^{n+1}.$ Let $\phi\in C^{\infty}_c(\mathbb{R}^{n+1})$ be compactly supported in $(-\pi,\pi)^{n+1}$ and such that $\phi=1$ on $[-2,2]^{n+1}.$ Let $\psi$ be the $2\pi$-periodic extension of the function
$$x\to |x|^{-\alpha(n+1)}\phi(x),\quad x\in[-\pi,\pi]^{n+1}.$$
Integral kernel of our operator is
$$(x,y)\to \chi_{B(0,R)}(x)\psi(x-y)\chi_{B(0,R)}(y),\quad x,y\in\mathbb{R}^{n+1}.$$
Identifying $[-\pi,\pi]^{n+1}$ with $\mathbb{T}^{n+1},$ we write our operator as $M_{\chi_{B(0,R)}}\hat{\psi}(\nabla_{\mathbb{T}^{n+1}})M_{\chi_{B(0,R)}}.$ Note that for $f\in L_2(\mathbb{T}^{n+1})$,
$$\hat{\psi}(\nabla_{\mathbb{T}^{n+1}})f(x)=\sum_{k\in\mathbb{Z}^{n+1}}\hat{\psi}(k)\hat{f}(k) e^{i\langle k,x\rangle},$$
and that $\{e^{i\langle k,x\rangle}\}_{k\in\mathbb{Z}^{n+1}}$ is an orthogonal basis of $L_2(\mathbb{T}^{n+1})$.
Thus,
\begin{align*}
\|V_K\|_{\mathcal{L}_{\frac{1}{1-\alpha},\infty}(L_2(B(0,R)))}&=\|M_{\chi_{B(0,R)}}\hat{\psi}(\nabla_{\mathbb{T}^{n+1}})M_{\chi_{B(0,R)}}\|_{\mathcal{L}_{\frac{1}{1-\alpha},\infty}(L_2(\mathbb{T}^{n+1}))}\\
&\leq\|\hat{\psi}(\nabla_{\mathbb{T}^{n+1}})\|_{\mathcal{L}_{\frac{1}{1-\alpha},\infty}(L_2(\mathbb{T}^{n+1}))}=\|\hat{\psi}\|_{l_{\frac1{1-\alpha},\infty}(\mathbb{Z}^{n+1})}.
\end{align*}
Since $\hat{\psi}\in l_{\frac1{1-\alpha},\infty}(\mathbb{Z}^{n+1}),$ the assertion follows.
\end{proof}

\begin{lemma}\label{rough estimate} Let $Q\subset\mathbb{R}^{n+1}$ be a cube and let $L\in L_{\infty}(Q\times Q).$ If
$$K(x,y)=\frac{L(x,y)}{|x-y|^{n-1}},\quad x,y\in Q,$$
and if $V_K$ is an integral operator with an integral kernel $K,$ then
$$\|V_K\|_{\mathcal{L}_{n+1}(L_2(Q))}\leq c_Q\|L\|_{L_{\infty}(Q\times Q)}$$
for some constant $c_Q>0$.
\end{lemma}
\begin{proof} If $n=1,$ then $K\in L_{\infty}(Q\times Q)\subset L_2(Q\times Q).$ Hence, $V_K\subset \mathcal{L}_2(L_2(Q))$ and
$$\|V_K\|_{\mathcal{L}_2(L_2(Q))}\leq m(Q)\|L\|_{L_{\infty}(Q\times Q)}.$$
This proves the assertion for $n=1.$
	
Suppose now $n\geq 2.$ Let
$$K_0(x,y)=|x-y|^{1-n},\quad x,y\in Q.$$
Using Lemma \ref{convolution operator lemma} with $\alpha=\frac{n-1}{n+1},$ we obtain $V_{K_0}\in\mathcal{L}_{\frac{n+1}{2},\infty}(L_2(Q)).$ Since $\frac{n+1}{2}< 2\lfloor\frac{n+1}{2}\rfloor,$ it follows that $V_{K_0}\in\mathcal{L}_{2\lfloor\frac{n+1}{2}\rfloor}(L_2(Q)).$ It follows from Lemma \ref{positive schur lemma} that
$$\|V_K\|_{\mathcal{L}_{2\lfloor\frac{n+1}{2}\rfloor}(L_2(Q))}\leq\|L\|_{\infty}\|V_{K_0}\|_{\mathcal{L}_{2\lfloor\frac{n+1}{2}\rfloor}(L_2(Q))}.$$
The assertion follows immediately.
\end{proof}

\begin{lemma}\label{mtb separable part lemma} Let $Q\subset\mathbb{R}^{n+1}$ be a cube. If $f\in L_{\infty}(\mathbb{R}^{n+1}),$ then
$$M_{\chi_Q}[\frac{\partial_k}{\Delta},M_f]M_{\chi_Q}\in (\mathcal{L}_{n+1,\infty}(L_2(\mathbb{R}^{n+1})))_0.$$
\end{lemma}
\begin{proof} {\bf Step 1:} Let $n\geq 2$ and let $f\in C^{\infty}_c(\mathbb{R}^{n+1}).$ By the Leibniz rule,
$$[\frac{\partial_k}{\Delta},M_f]=M_{\partial_kf}\Delta^{-1}-\frac{\partial_k}{\Delta}[\Delta,M_f]\Delta^{-1}=M_{\partial_kf}\Delta^{-1}+\sum_{l=1}^{n+1}\frac{\partial_k}{\Delta}[\partial_l^2,M_f]\Delta^{-1}.$$
Again by the Leibniz rule,
$$[\partial_l^2,M_f]=\partial_lM_{\partial_lf}+M_{\partial_lf}\partial_l=2\partial_lM_{\partial_lf}-M_{\partial_l^2f}.$$
Thus,
$$[\frac{\partial_k}{\Delta},M_f]=M_{\partial_kf}\Delta^{-1}+2\sum_{l=1}^{n+1}\frac{\partial_k\partial_l}{\Delta}M_{\partial_lf}\Delta^{-1}+\frac{\partial_k}{\Delta}M_{\Delta f}\Delta^{-1}.$$
Finally,
\begin{align}\label{rhsof}
&M_{\chi_Q}[\frac{\partial_k}{\Delta},M_f]M_{\chi_Q}=M_{\partial_kf\cdot \chi_Q}\Delta^{-1}M_{\chi_Q}\nonumber\\
&\hspace{1.0cm}+2\sum_{l=1}^{n+1}M_{\chi_Q}\frac{\partial_k\partial_l}{\Delta}\cdot M_{\partial_lf}\Delta^{-1}M_{\chi_Q}+M_{\chi_Q}\frac{\partial_k}{\Delta}\cdot M_{\Delta f}\Delta^{-1}M_{\chi_Q}.
\end{align}
By Lemma \ref{cwikel estimate in Euclidean setting} and $L_2$-boundedness of classical Riesz transform, for $g\in \{\partial_kf\cdot\chi_Q,\partial_lf,\Delta f\}$,
$$M_{g}\Delta^{-1}M_{\chi_Q}=M_{g}\Delta^{-\frac12}\cdot\Delta^{-\frac12}M_{\chi_Q}\in \mathcal{L}_{n+1,\infty}(L_2(\mathbb{R}^{n+1}))\cdot \mathcal{L}_{n+1,\infty}(L_2(\mathbb{R}^{n+1})),$$
$$M_{\chi_Q}\frac{\partial_k\partial_l}{\Delta}\in \mathcal{L}_{\infty}(L_2(\mathbb{R}^{n+1})),$$
$$M_{\chi_Q}\frac{\partial_k}{\Delta}=M_{\chi_Q}\Delta^{-\frac12}\cdot\frac{\partial_k}{\Delta^{\frac12}}\in \mathcal{L}_{n+1,\infty}(L_2(\mathbb{R}^{n+1}))\cdot \mathcal{L}_{\infty}(L_2(\mathbb{R}^{n+1})).$$
Hence, the right hand side on \eqref{rhsof} belongs to $\mathcal{L}_{\frac{n+1}{2},\infty}(L_2(\mathbb{R}^{n+1}))$ which is, clearly, a subset of $(\mathcal{L}_{n+1,\infty}(L_2(\mathbb{R}^{n+1})))_0.$ This completes the proof under the additional assumptions made in Step 1.

{\bf Step 2:} Let $n=1$ and let $f\in C^{\infty}_c(\mathbb{R}^{n+1})=C^{\infty}_c(\mathbb{R}^2).$ Integral kernel of the operator $M_{\chi_Q}[\frac{\partial_k}{\Delta},M_f]M_{\chi_Q}$ is (up to a constant factor)
$$(x,y)\to\frac{(y-x)_k(f(y)-f(x))}{|x-y|^2},\quad x,y\in Q.$$
Since $f\in C^{\infty}_c(\mathbb{R}^2),$ it follows that the integral kernel is bounded (thus, square integrable).  Hence, $$M_{\chi_Q}[\frac{\partial_k}{\Delta},M_f]M_{\chi_Q}\in\mathcal{L}_2(L_2(\mathbb{R}^2))\subset (\mathcal{L}_{n+1,\infty}(L_2(\mathbb{R}^2)))_0.$$ This completes the proof under the additional assumptions made in Step 2.

{\bf Step 3:} Consider now the general case. Fix a sequence $\{f_m\}_{m\geq1}\subset C^{\infty}_c(\mathbb{R}^{n+1})$ such that $\|f-f_m\|_{L_{2n+2}(Q)}\to0$ as $m\to\infty.$ By triangle inequality, we have
\begin{align*}
&\Big\|M_{\chi_Q}[\frac{\partial_k}{\Delta},M_f]M_{\chi_Q}-M_{\chi_Q}[\frac{\partial_k}{\Delta},M_{f_m}]M_{\chi_Q}\Big\|_{\mathcal{L}_{n+1,\infty}(L_2(\mathbb{R}^{n+1}))}\\
&\lesssim\Big\|M_{\chi_Q}\frac{\partial_k}{\Delta}M_{(f-f_m) \chi_Q}\Big\|_{\mathcal{L}_{n+1,\infty}(L_2(\mathbb{R}^{n+1}))}+\Big\|M_{(f-f_m)\chi_Q}\frac{\partial_k}{\Delta}M_{\chi_Q}\Big\|_{\mathcal{L}_{n+1,\infty}(L_2(\mathbb{R}^{n+1}))}.
\end{align*}
By Lemma \ref{cwikel estimate in Euclidean setting}, we have
\begin{align*}
&\Big\|M_{\chi_Q}[\frac{\partial_k}{\Delta},M_f]M_{\chi_Q}-M_{\chi_Q}[\frac{\partial_k}{\Delta},M_{f_m}]M_{\chi_Q}\Big\|_{\mathcal{L}_{n+1,\infty}(L_2(\mathbb{R}^{n+1}))}\\
&\leq C_nm(Q)^{\frac1{2n+2}}\|f-f_m\|_{L_{2n+2}(Q)}
\end{align*}
for some constant $C_n>0$. Hence, the left hand side tends to $0$ as $m\to\infty.$ The assertion follows now from Steps 1 and 2.
\end{proof}

\begin{lemma}\label{first distance lemma} Let $Q\subset\mathbb{R}^{n+1}$ be a cube. There exists a constant $C_n>0$ such that, for every $f\in L_{\infty}(\mathbb{R}^{n+1})\cap W^{1,n+1}(\mathbb{R}^{n+1}),$ we have
$$\|f\|_{\dot{W}^{1,n+1}(Q)}=C_n{\rm dist}_{\mathcal{L}_{n+1,\infty}(L_2(\mathbb{R}^{n+1}))}\Big(M_{\chi_Q}[R_k,M_f]M_{\chi_Q},(\mathcal{L}_{n+1,\infty}(L_2(\mathbb{R}^{n+1})))_0\Big)$$
for some constant $C_n>0$.
\end{lemma}
\begin{proof}
It follows from \cite[Proposition 8.6]{FSZ} that
$$\|f\|_{\dot{W}^{1,n+1}(Q)}=C_n \lim_{t\to\infty}t^{\frac1{n+1}}\mu_{B(L_2(\mathbb{R}^{n+1}))}(t,M_{\chi_Q}[R_k,M_f]M_{\chi_Q}).$$
As the right hand side is exactly the distance from the separable part, the assertion follows.
\end{proof}

\begin{lemma}\label{second distance lemma} Let $Q\subset\mathbb{R}^{n+1}$ be a cube. If $f\in L_{\infty}(\mathbb{R}^{n+1})$ is such that
$$M_{\chi_Q}[R_k,M_f]M_{\chi_Q}\in\mathcal{L}_{n+1,\infty}(L_2(\mathbb{R}^{n+1})),$$
then $f\in W^{1,n+1}(Q)$ and
$$\|f\|_{\dot{W}^{1,n+1}(Q)}\leq C_n{\rm dist}_{\mathcal{L}_{n+1,\infty}(L_2(\mathbb{R}^{n+1}))}\Big(M_{\chi_Q}[R_k,M_f]M_{\chi_Q},(\mathcal{L}_{n+1,\infty}(L_2(\mathbb{R}^{n+1})))_0\Big)$$
for some constant $C_n>0$.
\end{lemma}
\begin{proof} By translation and dilation, we may assume without loss of generality that $Q=[0,1]^{n+1}.$ Let $\phi_{\epsilon}\in C^{\infty}_c(\mathbb{R}^{n+1})$ be supported in $[\epsilon,1-\epsilon]^{n+1}$ and such that $\phi_{\epsilon}=1$ on $[2\epsilon,1-2\epsilon]^{n+1}.$ We have
$$M_{\phi_{\epsilon}}[R_k,M_f]M_{\phi_{\epsilon}}\in\mathcal{L}_{n+1,\infty}(L_2(\mathbb{R}^{n+1})).$$
By the Leibniz rule,
$$[R_k,M_{f\phi_{\epsilon}^2}]=[R_k,M_{\phi_{\epsilon}}] M_{f\phi_{\epsilon}}+M_{\phi_{\epsilon}}[R_k,M_f]M_{\phi_{\epsilon}}+M_{f\phi_{\epsilon}} [R_k,M_{\phi_{\epsilon}}].$$
Since $\phi_{\epsilon}\in C^{\infty}_c(\mathbb{R}^{n+1}),$ it follows that
$$[R_k,M_{\phi_{\epsilon}}]\in\mathcal{L}_{n+1,\infty}(L_2(\mathbb{R}^{n+1})).$$
Since $f$ is bounded, it follows that
$$[R_k,M_{f\phi_{\epsilon}^2}]\in\mathcal{L}_{n+1,\infty}(L_2(\mathbb{R}^{n+1})).$$
The latter inclusion is purely quantitative, as we do not have any reasonable control of the norm. By \cite[Theorem 1]{LMSZ}, $f\phi_{\epsilon}^2\in \dot{W}^{1,n+1}(\mathbb{R}^{n+1}).$ Since the function $f\phi_{\epsilon}^2$ is bounded and supported on $[0,1]^{n+1},$ it follows that $f\phi_{\epsilon}^2\in L_{\infty}(\mathbb{R}^{n+1})\cap W^{1,n+1}(\mathbb{R}^{n+1}).$

It is immediate that
$$M_{\chi_{[2\epsilon,1-2\epsilon]^{n+1}}}[R_k,M_f]M_{\chi_{[2\epsilon,1-2\epsilon]^{n+1}}}=M_{\chi_{[2\epsilon,1-2\epsilon]^{n+1}}}[R_k,M_{f\phi_{\epsilon}^2}]M_{\chi_{[2\epsilon,1-2\epsilon]^{n+1}}}.$$
Thus,
\begin{align*}
&{\rm dist}_{\mathcal{L}_{n+1,\infty}(L_2(\mathbb{R}^{n+1}))}\Big(M_{\chi_Q}[R_k,M_f]M_{\chi_Q},(\mathcal{L}_{n+1,\infty}(L_2(\mathbb{R}^{n+1})))_0\Big)\\
&\geq {\rm dist}_{\mathcal{L}_{n+1,\infty}(L_2(\mathbb{R}^{n+1}))}\Big(M_{\chi_{[2\epsilon,1-2\epsilon]^{n+1}}}[R_k,M_f]M_{\chi_{[2\epsilon,1-2\epsilon]^{n+1}}},(\mathcal{L}_{n+1,\infty}(L_2(\mathbb{R}^{n+1})))_0\Big)\\
&={\rm dist}_{\mathcal{L}_{n+1,\infty}(L_2(\mathbb{R}^{n+1}))}\Big(M_{\chi_{[2\epsilon,1-2\epsilon]^{n+1}}}[R_k,M_{f\phi_{\epsilon}^2}]M_{\chi_{[2\epsilon,1-2\epsilon]^{n+1}}},(\mathcal{L}_{n+1,\infty}(L_2(\mathbb{R}^{n+1})))_0\Big).
\end{align*}
Since $f\phi_{\epsilon}^2\in L_{\infty}(\mathbb{R}^{n+1})\cap W^{1,n+1}(\mathbb{R}^{n+1}),$ it follows from Lemma \ref{first distance lemma} that
\begin{align*}
&C_n{\rm dist}_{\mathcal{L}_{n+1,\infty}(L_2(\mathbb{R}^{n+1}))}\Big(M_{\chi_Q}[R_k,M_f]M_{\chi_Q},(\mathcal{L}_{n+1,\infty}(L_2(\mathbb{R}^{n+1})))_0\Big)\\
&\geq \|f\phi_{\epsilon}^2\|_{\dot{W}^{1,n+1}([2\epsilon,1-2\epsilon]^{n+1})}=\|f\|_{\dot{W}^{1,n+1}([2\epsilon,1-2\epsilon]^{n+1})}.
\end{align*}
Passing $\epsilon\downarrow0,$ we complete the proof.
\end{proof}
The following proposition provides a suitable local approximation of the Bessel--Riesz commutators.
\begin{proposition}\label{mtb main lemma} Let $f\in L_{\infty}(\mathbb{R}^{n+1})$ and let $Q$ be a cube compactly supported in $\mathbb{R}^{n+1}_+.$ For $1\leq k\leq n+1,$ we have
\begin{align*}
&M_{\chi_Q}M_{x_{n+1}^{\lambda}}E[R_{\lambda,k},M_{E^{\ast}f}]E^{\ast}M_{x_{n+1}^{-\lambda}}M_{\chi_Q}\\
&\hspace{1.0cm}-\kappa^{[3]}_{n,\lambda}F_{2,0}(0)M_{\chi_Q}[R_k,M_f]M_{\chi_Q}\in (\mathcal{L}_{n+1,\infty}(L_2(\mathbb{R}^{n+1})))_0.
\end{align*}
\end{proposition}
\begin{proof} By Lemma \ref{Bessel--Riesz kernel representation lemma}, integral kernel of the operator $M_{x_{n+1}^{\lambda}}[R_{\lambda,k},M_{E^{\ast}f}]M_{x_{n+1}^{-\lambda}}$ is given by the formula
\begin{align*}
(x,y)&\to\frac{\kappa^{[2]}_{n,\lambda}(y-x)_{k}(f(y)-f(x))}{|x-y|^{n+2}}\cdot F_{2,0}(H(x,y))\\
&\hspace{-0.8cm}+\kappa^{[2]}_{n,\lambda}\delta_{k,n+1}\sum_{l=1}^{n+1} h_l(x,y)a(x,y) \frac{(y-x)_{n+1}(f(y)-f(x))}{|x-y|^{n+2}}\cdot F_{1,1}(H(x,y))\\
&\hspace{-0.8cm}-\kappa^{[2]}_{n,\lambda}\delta_{k,n+1}\sum_{l=1}^{n+1} h_l(x,y)h_{n+1}(x,y)b(x,y) \frac{(y-x)_{n+1}(f(y)-f(x))}{|x-y|^{n+2}}\cdot F_{2,1}(H(x,y)).
\end{align*}

As $Q$ is compactly supported in $\mathbb{R}^{n+1}_+,$ we apply Taylor's formula to deduce that for any $(j,l)\in \{(2,0),(1,1),(2,1)\}$,
$$F_{j,l}(\frac{|x-y|}{(x_{n+1}y_{n+1})^{\frac12}})=F_{j,l}(0)+F_{j,l}^{(1)}(0)\frac{|x-y|}{(x_{n+1}y_{n+1})^{\frac12}}+O(|x-y|^2),\quad x,y\in Q.$$
Taking into account that $F_{1,1}(0)=F_{2,1}(0)=0$, we conclude that integral kernel of the operator $$M_{\chi_Q} M_{x_{n+1}^{\lambda}}E[R_{\lambda,k},M_{E^{\ast}f}]E^{\ast}M_{x_{n+1}^{-\lambda}} M_{\chi_Q}$$ is given by the formula
\begin{align*}
(x,y)&\to\frac{\kappa^{[2]}_{n,\lambda}(y-x)_k(f(y)-f(x))}{|x-y|^{n+2}}\chi_Q(x)\chi_Q(y)\cdot F_{2,0}(0)\\
&\hspace{-0.5cm}+\frac{\kappa^{[2]}_{n,\lambda}(y-x)_k(f(y)-f(x))}{|x-y|^{n+1}(x_{n+1}y_{n+1})^{\frac12}}\chi_Q(x)\chi_Q(y)\cdot F_{2,0}^{(1)}(0)\\
&\hspace{-0.5cm}+\kappa^{[2]}_{n,\lambda}\delta_{k,n+1}\sum_{l=1}^{n+1} h_l(x,y)a(x,y) \frac{(y-x)_{n+1}(f(y)-f(x))}{|x-y|^{n+1}(x_{n+1}y_{n+1})^{\frac12}}\chi_Q(x)\chi_Q(y)\cdot F_{1,1}^{(1)}(0)\\
&\hspace{-0.5cm}-\kappa^{[2]}_{n,\lambda}\delta_{k,n+1}\sum_{l=1}^{n+1} h_l(x,y)h_{n+1}(x,y)b(x,y)\times \\ &\hspace{1.0cm}\times\frac{(y-x)_{n+1}(f(y)-f(x))}{|x-y|^{n+1}(x_{n+1}y_{n+1})^{\frac12}}\chi_Q(x)\chi_Q(y)\cdot F_{2,1}^{(1)}(0)\\
&\hspace{-0.5cm}+O(|x-y|^{1-n}),\quad x,y\in Q.
\end{align*}
The first four summands on the right hand side are the integral kernels of
$$\kappa^{[3]}_{n,\lambda}F_{2,0}(0)M_{\chi_Q}[R_k,M_f]M_{\chi_Q},$$
$$C_{n,\lambda}F_{2,0}^{(1)}(0)M_{\chi_Q}M_{x_{n+1}^{-\frac12}}[\frac{\partial_k}{\Delta},M_f]M_{x_{n+1}^{-\frac12}}M_{\chi_Q},$$
$$C_{n,\lambda}F_{1,1}^{(1)}(0)\delta_{k,n+1}\sum_{l=1}^{n+1}\Big(\mathfrak{S}_{h_l}\circ\mathfrak{S}_a\Big)\Big(M_{\chi_Q}M_{x_{n+1}^{-\frac12}}[\frac{\partial_{n+1}}{\Delta},M_f]M_{x_{n+1}^{-\frac12}}M_{\chi_Q}\Big),$$
$$C_{n,\lambda}F_{2,1}^{(1)}(0)\delta_{k,n+1}\sum_{l=1}^{n+1}\Big(\mathfrak{S}_{h_l}\circ\mathfrak{S}_{h_{n+1}}\circ\mathfrak{S}_b\Big)\Big(M_{\chi_Q}M_{x_{n+1}^{-\frac12}}[\frac{\partial_{n+1}}{\Delta},M_f]M_{x_{n+1}^{-\frac12}}M_{\chi_Q}\Big),$$
respectively, for some constant $C_{n,\lambda}>0$. By Lemma \ref{rough estimate}, integral operator associated with the fifth summand belongs to  $\mathcal{L}_{n+1}(L_2(\mathbb{R}^{n+1}))\subset (\mathcal{L}_{n+1,\infty}(L_2(\mathbb{R}^{n+1})))_0.$

Since $Q$ is a cube in $\mathbb{R}^{n+1}_+,$ it follows from Lemma \ref{mtb separable part lemma} that
 $$M_{\chi_Q}M_{x_{n+1}^{-\frac12}}[\frac{\partial_k}{\Delta},M_f]M_{x_{n+1}^{-\frac12}}M_{\chi_Q}\in (\mathcal{L}_{n+1,\infty}(L_2(\mathbb{R}^{n+1})))_0.$$
This, in combination with Proposition \ref{standard schur lemma}, also implies that
$$\sum_{l=1}^{n+1}\Big(\mathfrak{S}_{h_l}\circ\mathfrak{S}_a\Big)\Big(M_{\chi_Q}M_{x_{n+1}^{-\frac12}}[\frac{\partial_{n+1}}{\Delta},M_f]M_{x_{n+1}^{-\frac12}}M_{\chi_Q}\Big)\in (\mathcal{L}_{n+1,\infty}(L_2(\mathbb{R}^{n+1})))_0,$$
$$\sum_{l=1}^{n+1}\Big(\mathfrak{S}_{h_l}\circ\mathfrak{S}_{h_{n+1}}\circ\mathfrak{S}_b\Big)\Big(M_{\chi_Q}M_{x_{n+1}^{-\frac12}}[\frac{\partial_{n+1}}{\Delta},M_f]M_{x_{n+1}^{-\frac12}}M_{\chi_Q}\Big)\in (\mathcal{L}_{n+1,\infty}(L_2(\mathbb{R}^{n+1})))_0.$$
This completes the proof.
\end{proof}

\begin{proof}[Proof of lower bound in Theorem \ref{main theorem}] It follows from Lemma \ref{weight} that for $f\in L_{\infty}(\mathbb{R}^{n+1}_+),$
\begin{align*}
&\|[R_{\lambda,k},M_f]\|_{\mathcal{L}_{n+1,\infty}(L_2(\mathbb{R}^{n+1}_+,m_{\lambda}))}\\
&\geq {\rm dist}_{\mathcal{L}_{n+1,\infty}(L_2(\mathbb{R}^{n+1}_+,m_{\lambda}))}([R_{\lambda,k},M_f],(\mathcal{L}_{n+1,\infty}(L_2(\mathbb{R}^{n+1}_+,m_{\lambda})))_0)\\
&={\rm dist}_{\mathcal{L}_{n+1,\infty}(L_2(\mathbb{R}^{n+1}_+))}(M_{x_{n+1}^{\lambda}}[R_{\lambda,k},M_f]M_{x_{n+1}^{-\lambda}},(\mathcal{L}_{n+1,\infty}(L_2(\mathbb{R}^{n+1}_+)))_0)\\
&\geq {\rm dist}_{\mathcal{L}_{n+1,\infty}(L_2(\mathbb{R}^{n+1}_+))}(M_{\chi_Q} M_{x_{n+1}^{\lambda}}[R_{\lambda,k},M_f]M_{x_{n+1}^{-\lambda}} M_{\chi_Q},(\mathcal{L}_{n+1,\infty}(L_2(\mathbb{R}^{n+1}_+)))_0)\\
&={\rm dist}_{\mathcal{L}_{n+1,\infty}(L_2(\mathbb{R}^{n+1}))}(M_{\chi_Q} M_{x_{n+1}^{\lambda}}E[R_{\lambda,k},M_f]E^{\ast}M_{x_{n+1}^{-\lambda}} M_{\chi_Q},(\mathcal{L}_{n+1,\infty}(L_2(\mathbb{R}^{n+1})))_0)
\end{align*}
for every cube $Q$ compactly supported in $\mathbb{R}^{n+1}_+.$	

Suppose $1\leq k\leq n+1.$ It follows from Proposition \ref{mtb main lemma} that
\begin{align*}
&\|[R_{\lambda,k},M_f]\|_{\mathcal{L}_{n+1,\infty}(L_2(\mathbb{R}^{n+1}_+,m_{\lambda}))}\\
&\geq \kappa^{[3]}_{n,\lambda}F_{2,0}(0){\rm dist}_{\mathcal{L}_{n+1,\infty}(L_2(\mathbb{R}^{n+1}))}(M_{\chi_Q} [R_k,M_{Ef}] M_{\chi_Q},(\mathcal{L}_{n+1,\infty}(L_2(\mathbb{R}^{n+1})))_0).
\end{align*}
Combining this with Lemma \ref{second distance lemma}, we conclude that there is a constant $C_{n,\lambda}>0$ such that
$$\|[R_{\lambda,k},M_f]\|_{\mathcal{L}_{n+1,\infty}(L_2(\mathbb{R}^{n+1}_+,m_{\lambda}))}\geq C_{n,\lambda} \|f\|_{\dot{W}^{1,n+1}(Q)}.$$
Taking the supremum over all cubes $Q$ compactly supported in $\mathbb{R}^{n+1}_+,$ we complete the proof.	
\end{proof}

\section{Proof of the spectral asymptotic}\label{spect}
\setcounter{equation}{0}

This section is devoted to establishing the spectral asymptotic formula for the Bessel--Riesz commutator. The key tool is to establish a suitable approximation of the commutator.

\begin{lemma}\label{abcde}
Let $1<p<\infty$. Assume that $V\in \mathcal{L}_{p,\infty}(L_2(\mathbb{R}^{n+1}_+))$ be compactly supported in $\mathbb{R}^{n+1}_+.$ If
$[M_{x_l},V]\in (\mathcal{L}_{p,\infty}(L_2(\mathbb{R}^{n+1}_+)))_0$ for every $1\leq l\leq n,$ then $\mathfrak{S}_H(V)\in (\mathcal{L}_{p,\infty}(L_2(\mathbb{R}^{n+1}_+)))_0.$
\end{lemma}
\begin{proof} Set
$$\Theta_l(x)=\frac{|x-y|}{\sum_{k=1}^{n+1}|x_k-y_k|}\cdot {\rm sgn}(x_l-y_l),\quad x,y\in\mathbb{R}^{n+1},\quad 1\leq l\leq n+1.$$
Let $V$ be supported on a compact set $Q\subset \mathbb{R}^{n+1}_+.$ We have
$$H(x,y)\chi_Q(x)\chi_Q(y)=x_{n+1}^{-\frac12}\chi_Q(x)\cdot y_{n+1}^{-\frac12}\chi_Q(y)\cdot  \sum_{l=1}^{n+1}\Theta_l(x,y)\cdot (x-y)_l.$$
Thus,
$$\mathfrak{S}_H(V)=\sum_{l=1}^{n+1}\mathfrak{S}_{\Theta_l}\Big(M_{x_{n+1}^{-\frac12}\chi_Q} [M_{x_l},V] M_{x_{n+1}^{-\frac12}\chi_Q}\Big).$$
By assumption, the argument of Schur multiplier $\mathfrak{S}_{\Theta_l}$ belongs to $(\mathcal{L}_{p,\infty}(L_2(\mathbb{R}^{n+1}_+)))_0$ for every $1\leq l\leq n+1.$ Since Schur multiplier $\mathfrak{S}_{\Theta_l}$ sends $(\mathcal{L}_{p,\infty}(L_2(\mathbb{R}^{n+1}_+)))_0$ to itself for every $1\leq l\leq n+1,$ the assertion follows.
\end{proof}

\begin{lemma}\label{criterion lemma}
Let $1<p<\infty$. Assume that $V\in \mathcal{L}_{p,\infty}(L_2(\mathbb{R}^{n+1}_+))$ and $(V_j)_{j\geq1}\subset \mathcal{L}_{p,\infty}(L_2(\mathbb{R}^{n+1}_+))$ such that
\begin{enumerate}[{\rm (i)}]
\item\label{cla} ${\rm dist}_{\mathcal{L}_{p,\infty}(L_2(\mathbb{R}^{n+1}_+))}(V_j-V,(\mathcal{L}_{p,\infty}(L_2(\mathbb{R}^{n+1}_+)))_0)\to 0$ as $j\to\infty;$
\item\label{clb} $[M_{x_l},V_j]\in (\mathcal{L}_{p,\infty}(L_2(\mathbb{R}^{n+1}_+)))_0$ for every $j\geq1$ and for every $1\leq l\leq n;$
\item\label{clc} for every $j\geq 1,$ the operator $V_j$ is compactly supported in $\mathbb{R}^{n+1}_+,$
\end{enumerate}	
then for $(k,l)\in \{(2,0),(1,1),(2,1)\}$, we have
$$\mathfrak{S}_{F_{k,l}\circ H}(V)-F_{k,l}(0)V\in(\mathcal{L}_{p,\infty}(L_2(\mathbb{R}^{n+1}_+)))_0.$$
\end{lemma}
\begin{proof} We write
$$\mathfrak{S}_{F_{k,l}\circ H}(V_j)=F_{k,l}(0)V_j+\mathfrak{S}_{G_{k,l}\circ H}(\mathfrak{S}_H(V_j)),$$
where
$$G_{k,l}(x):=\frac{F_{k,l}(x)-F_{k,l}(+0)}{x},\quad x\in(0,\infty).$$
Applying Lemma \ref{abcde} (whose assumptions are satisfied due to \eqref{clb} and \eqref{clc}) to the operator $V_j,$ we conclude that
$$\mathfrak{S}_H(V_j)\in (\mathcal{L}_{p,\infty}(L_2(\mathbb{R}^{n+1}_+)))_0,\quad j\geq 1.$$
By Theorem \ref{fkl are smooth theorem}, the function $G_{k,l}$ satisfies the assumptions in Theorem \ref{main schur theorem}. By Theorem \ref{main schur theorem},  $\mathfrak{S}_{G_{k,l}\circ H}:(\mathcal{L}_{p,\infty}(L_2(\mathbb{R}^{n+1}_+)))_0\to (\mathcal{L}_{p,\infty}(L_2(\mathbb{R}^{n+1}_+)))_0.$  Hence,
\begin{equation}\label{cl eq0}
\mathfrak{S}_{F_{k,l}\circ H}(V_j)-F_{k,l}(0)V_j\in (\mathcal{L}_{p,\infty}(L_2(\mathbb{R}^{n+1}_+)))_0,\quad j\geq 1.
\end{equation}
We have
\begin{align*}
&{\rm dist}_{\mathcal{L}_{p,\infty}(L_2(\mathbb{R}^{n+1}_+))}(\mathfrak{S}_{F_{k,l}\circ H}(V_j)-\mathfrak{S}_{F_{k,l}\circ H}(V),(\mathcal{L}_{p,\infty}(L_2(\mathbb{R}^{n+1}_+)))_0)\\
&=\inf_{A\in \mathcal{L}_p(L_2(\mathbb{R}^{n+1}_+))}\Big\|\mathfrak{S}_{F_{k,l}\circ H}(V_j)-\mathfrak{S}_{F_{k,l}\circ H}(V)-A\Big\|_{\mathcal{L}_{p,\infty}(L_2(\mathbb{R}^{n+1}_+))}\\
&\leq\inf_{\substack{A=\mathfrak{S}_{F_{k,l}\circ H}(B)\\ B\in \mathcal{L}_p(L_2(\mathbb{R}^{n+1}_+))}}\Big\|\mathfrak{S}_{F_{k,l}\circ H}(V_j)-\mathfrak{S}_{F_{k,l}\circ H}(V)-A\Big\|_{\mathcal{L}_{p,\infty}(L_2(\mathbb{R}^{n+1}_+))}\\
&=\inf_{B\in \mathcal{L}_p(L_2(\mathbb{R}^{n+1}_+))}\Big\|\mathfrak{S}_{F_{k,l}\circ H}(V_j)-\mathfrak{S}_{F_{k,l}\circ H}(V)-\mathfrak{S}_{F_{k,l}\circ H}(B)\Big\|_{\mathcal{L}_{p,\infty}(L_2(\mathbb{R}^{n+1}_+))}\\
&=\inf_{B\in \mathcal{L}_p(L_2(\mathbb{R}^{n+1}_+))}\Big\|\mathfrak{S}_{F_{k,l}\circ H}(V_j-V-B)\Big\|_{\mathcal{L}_{p,\infty}(L_2(\mathbb{R}^{n+1}_+))}\\
&\leq\Big\|\mathfrak{S}_{F_{k,l}\circ H}\Big\|_{\mathcal{L}_{p,\infty}(L_2(\mathbb{R}^{n+1}_+))\circlearrowleft}\cdot\inf_{B\in \mathcal{L}_p(L_2(\mathbb{R}^{n+1}_+))}\Big\|V_j-V-B\Big\|_{\mathcal{L}_{p,\infty}(L_2(\mathbb{R}^{n+1}_+))}\\
&\leq\Big\|\mathfrak{S}_{F_{k,l}\circ H}\Big\|_{\mathcal{L}_{p,\infty}(L_2(\mathbb{R}^{n+1}_+))\circlearrowleft}\cdot{\rm dist}_{\mathcal{L}_{p,\infty}(L_2(\mathbb{R}^{n+1}_+))}(V_j-V,(\mathcal{L}_{p,\infty}(L_2(\mathbb{R}^{n+1}_+)))_0).
\end{align*}
By \eqref{cla}, the right hand side tends to $0$ as $j\to\infty.$ Hence, so does the left hand side. It follows now from \eqref{cl eq0} that
$${\rm dist}_{\mathcal{L}_{p,\infty}(L_2(\mathbb{R}^{n+1}_+))}(F_{k,l}(0)\cdot V_j-\mathfrak{S}_{F_{k,l}\circ H}(V),(\mathcal{L}_{p,\infty}(L_2(\mathbb{R}^{n+1}_+)))_0)\to0$$
as $j\to\infty.$ The assertion follows now from \eqref{cla}.
\end{proof}

The following lemma is well known. We provide a short proof for convenience of the reader.

\begin{lemma}\label{iii verification first lemma} If $1\leq k,l\leq n+1$ and $f\in C^{\infty}_c(\mathbb{R}^{n+1}),$ then the operators
$$[(1+\Delta)^{-\frac12},M_f],[\partial_k\partial_l(1+\Delta)^{-\frac32},M_f],\,{\rm and}\ [R_kR_m(1+\Delta)^{-\frac12},M_{f}]$$
belong to $\mathcal{L}_{\frac{n+1}{2},\infty}(L_2(\mathbb{R}^{n+1}))$.
\end{lemma}
\begin{proof} Fix $\psi\in C^{\infty}_c(\mathbb{R}^{n+1})$ such that $f=f\psi.$ Using Leibniz rule, we write
\begin{align*}
[(1+\Delta)^{-\frac12},M_f]
&=[(1+\Delta)^{-\frac12},M_{f}]M_{\psi}+M_{f}[(1+\Delta)^{-\frac12},M_{\psi}]\\
&=[(1+\Delta)^{-\frac12},M_{f}](1+\Delta)\cdot (1+\Delta)^{-1}M_{\psi}\\
&\hspace{1.0cm}+M_{f}(1+\Delta)^{-1}\cdot (1+\Delta)[(1+\Delta)^{-\frac12},M_{\psi}].
\end{align*}
Since $(1+\Delta)^{-\frac12}$ is a pseudo-differential operator of order $-1$ and since $M_{f}$ and $M_{\psi}$ are pseudo-differential operators of order $0,$ it follows that $[(1+\Delta)^{-\frac12},M_{f}]$ and $[(1+\Delta)^{-\frac12},M_{\psi}]$ are pseudo-differential operators of order $-2.$ Thus,
$$[(1+\Delta)^{-\frac12},M_{f}](1+\Delta)\mbox{ and }(1+\Delta)[(1+\Delta)^{-\frac12},M_{\psi}]$$
are pseudo-differential operators of order $0$ and are, therefore, bounded. The first assertion follows now from Lemma \ref{cwikel estimate in Euclidean setting}. The proof of the second assertion follows {\it mutatis mutandi}.

To see the third assertion, we write
\begin{align*}
&[R_kR_m(1+\Delta)^{-\frac12},M_{f}]\\
&=[\frac{\partial_k\partial_m}{(1+\Delta)^{\frac32}},M_f]+[\frac{\partial_k\partial_m}{\Delta(1+\Delta)^{\frac32}},M_f]\\
&=[\frac{\partial_k\partial_m}{(1+\Delta)^{\frac32}},M_f]+R_kR_m\cdot (1+\Delta)^{-\frac32}M_f-M_f(1+\Delta)^{-\frac32}\cdot R_kR_m.
\end{align*}
That commutator on the right side falls into $\mathcal{L}_{\frac{n+1}{2},\infty}(L_2(\mathbb{R}^{n+1}))$ can be proved repeating the argument in the first paragraph {\it mutatis mutandi}. That the other summands on the right hand side fall into $\mathcal{L}_{\frac{n+1}{2},\infty}(L_2(\mathbb{R}^{n+1}))$ follows from Cwikel estimates.
\end{proof}

\begin{lemma}\label{iii verification second lemma} If $1\leq k,l\leq n+1$ and $f\in C^{\infty}_c(\mathbb{R}^{n+1}),$ then
$$M_{\chi_Q}[\Delta^{-\frac12},M_f]M_{\chi_Q},M_{\chi_Q}[\partial_k\partial_l\Delta^{-\frac32},M_f]M_{\chi_Q}\in \mathcal{L}_{\frac{n+1}{2},\infty}(L_2(\mathbb{R}^{n+1}))$$
for every compact set $Q\subset\mathbb{R}^{n+1}.$
\end{lemma}
\begin{proof} We write
$$M_{\chi_Q}[\Delta^{-\frac12},M_f]M_{\chi_Q}=M_{\chi_Q}[(1+\Delta)^{-\frac12},M_f]M_{\chi_Q}+[M_{\chi_Q}g(\sqrt{\Delta})M_{\chi_Q},M_f],$$
$$M_{\chi_Q}[\partial_k\partial_l\Delta^{-\frac32},M_f]M_{\chi_Q}=M_{\chi_Q}[\partial_k\partial_l(1+\Delta)^{-\frac32},M_f]M_{\chi_Q}-[M_{\chi_Q}g_{k,l}(\sqrt{\Delta})M_{\chi_Q},M_f],$$		
where
$$g(x):=|x|^{-1}-(1+|x|^2)^{-\frac12},\quad g_{k,l}(x):=x_kx_l(|x|^{-3}-(1+|x|^2)^{-\frac32}),\quad x\in\mathbb{R}^{n+1}.$$
Applying the Abstract Cwikel Estimate (see Theorem 1.5.10 in \cite{LSZ2}) in the standard setting of $\mathbb{R}^{n+1},$ we obtain
$$M_{\chi_Q}g(\nabla)M_{\chi_Q}\prec\prec 160\chi_Q\otimes g,\quad M_{\chi_Q}g_{k,l}(\nabla)M_{\chi_Q}\prec\prec 160\chi_Q\otimes g_{k,l}.$$
Here we recall in Section \ref{Spdef} that $ A \prec\prec B$ means $A$ is submajorized by $B$.
Since $\chi_Q\otimes g,\chi_Q\otimes g_{k,l}\in (L_1+L_{\frac{n+1}{2}})(\mathbb{R}^{n+1}\times\mathbb{R}^{n+1}),$ it follows that
$$M_{\chi_Q}g(\nabla)M_{\chi_Q},M_{\chi_Q}g_{k,l}(\nabla)M_{\chi_Q}\in\mathcal{L}_{\frac{n+1}{2}}(L_2(\mathbb{R}^{n+1})).$$
The assertion follows now from Lemma \ref{iii verification first lemma}.
\end{proof}

\begin{lemma}\label{ii pre-verification lemma} If $Q_j=[0,j]^n\times[\frac1j,j]$ and $f\in C^{\infty}_c(\mathbb{R}^{n+1}),$ then
$$M_{f\cdot\chi_{Q_j}}(1+\Delta)^{-\frac12}\to M_{f\cdot\chi_{\mathbb{R}^{n+1}_+}}(1+\Delta)^{-\frac12},\ {\rm as}\ j\rightarrow \infty,$$
in $\mathcal{L}_{n+1,\infty}(L_2(\mathbb{R}^{n+1})).$
\end{lemma}
\begin{proof} If $n\geq 2,$ then Lemma \ref{cwikel estimate in Euclidean setting} yields
$$\Big\|M_{f\cdot\chi_{\mathbb{R}^{n+1}_+\backslash Q_j}}(1+\Delta)^{-\frac12}\Big\|_{\mathcal{L}_{n+1,\infty}(L_2(\mathbb{R}^{n+1}))}\lesssim\|f\cdot\chi_{\mathbb{R}^{n+1}_+\backslash Q_j}\|_{L_{n+1}(\mathbb{R}^{n+1})}\to0,\quad j\to\infty.$$

Next, we consider the case $n=1.$ Fix $\phi\in C^{\infty}_c(\mathbb{R}^2)$ such that $f=f\phi.$ We have
$$M_{f\cdot\chi_{\mathbb{R}^2_+\backslash Q_j}}(1+\Delta)^{-\frac12}=M_{f\cdot\chi_{\mathbb{R}^2_+\backslash Q_j}}(1+\Delta)^{-\frac12}M_{\phi}+M_{f\cdot\chi_{\mathbb{R}^2_+\backslash Q_j}}\cdot[M_{\phi},(1+\Delta)^{-\frac12}].$$
By triangle inequality,
\begin{align*}
&\Big\|M_{f\cdot\chi_{\mathbb{R}^2_+\backslash Q_j}}(1+\Delta)^{-\frac12}\Big\|_{\mathcal{L}_{2,\infty}(L_2(\mathbb{R}^2))}\\
&\lesssim\Big\|M_{f\cdot\chi_{\mathbb{R}^2_+\backslash Q_j}}(1+\Delta)^{-\frac12}M_{\phi}\Big\|_{\mathcal{L}_{2,\infty}(L_2(\mathbb{R}^2))}+\Big\|M_{f\cdot\chi_{\mathbb{R}^2_+\backslash Q_j}}\cdot[M_{\phi},(1+\Delta)^{-\frac12}]\Big\|_{\mathcal{L}_{2,\infty}(L_2(\mathbb{R}^2))}\\
&\lesssim\Big\|M_{f\cdot\chi_{\mathbb{R}^2_+\backslash Q_j}}(1+\Delta)^{-\frac12}M_{\phi}\Big\|_{\mathcal{L}_{2,\infty}(L_2(\mathbb{R}^2))}+\Big\|M_{f\cdot\chi_{\mathbb{R}^2_+\backslash Q_j}}\cdot[M_{\phi},(1+\Delta)^{-\frac12}]\Big\|_{\mathcal{L}_2(L_2(\mathbb{R}^2))}\\
&\lesssim\Big\|M_{f\cdot\chi_{\mathbb{R}^2_+\backslash Q_j}}(1+\Delta)^{-\frac14}\Big\|_{\mathcal{L}_{4,\infty}(L_2(\mathbb{R}^2))}\cdot \Big\|(1+\Delta)^{-\frac14}M_{\phi}\Big\|_{\mathcal{L}_{4,\infty}(L_2(\mathbb{R}^2))}\\
&\hspace{1.0cm}+\Big\|M_{f\cdot\chi_{\mathbb{R}^2_+\backslash Q_j}}(1+\Delta)^{-1}\Big\|_{\mathcal{L}_2(L_2(\mathbb{R}^2))}\cdot\Big\|(1+\Delta)[M_{\phi},(1+\Delta)^{-\frac12}]\Big\|_{\mathcal{L}_{\infty}(L_2(\mathbb{R}^2))}.
\end{align*}
Using Lemma \ref{cwikel estimate in Euclidean setting}, we obtain
\begin{align*}
&\Big\|M_{f\cdot\chi_{\mathbb{R}^2_+\backslash Q_j}}(1+\Delta)^{-\frac12}\Big\|_{\mathcal{L}_{2,\infty}(L_2(\mathbb{R}^2))}\\
&\leq c_{\phi}(\|f\cdot\chi_{\mathbb{R}^2_+\backslash Q_j}\|_{L_4(\mathbb{R}^2)}+\|f\cdot\chi_{\mathbb{R}^2_+\backslash Q_j}\|_{L_2(\mathbb{R}^2)})\to0,\quad {\rm as}\ j\to\infty.
\end{align*}
This ends the proof of Lemma \ref{ii pre-verification lemma}.
\end{proof}

\begin{lemma}\label{ii verification lemma} If $1\leq k\leq n+1$, $Q_j=[0,j]^n\times[\frac1j,j]$ and $f\in C^{\infty}_c(\mathbb{R}^{n+1}),$ then
$$M_{\chi_{Q_j}}[R_k,M_f]M_{\chi_{Q_j}}\to M_{\chi_{\mathbb{R}^{n+1}_+}}[R_k,M_f]M_{\chi_{\mathbb{R}^{n+1}_+}},\ {\rm as}\ j\rightarrow \infty,$$
in the semi-norm ${\rm dist}_{\mathcal{L}_{n+1,\infty}(L_2(\mathbb{R}^{n+1}))}(\cdot,(\mathcal{L}_{n+1,\infty}(L_2(\mathbb{R}^{n+1})))_0).$
\end{lemma}
\begin{proof} Applying \cite[Theorem 6.3.1]{LSZ2}, we write
\begin{align}\label{equ00}
[R_k,M_f]=i\left(M_{\partial_k f}-R_k\sum_{m=1}^{n+1}R_mM_{\partial_m f}\right)(1+\Delta)^{-\frac12}+E,
\end{align}
with $E\in (\mathcal{L}_{n+1,\infty}(L_2(\mathbb{R}^{n+1})))_0.$ Fix $\psi\in C^{\infty}_c(\mathbb{R}^{n+1})$ such that $f=f\psi.$ Fix $\phi\in C^{\infty}_c(\mathbb{R}^{n+1})$ such that $\psi=\psi\phi.$ By Lemma \ref{iii verification first lemma},
$$[M_{\partial_mf},(1+\Delta)^{-\frac12}],[M_{\partial_mf},R_kR_m(1+\Delta)^{-\frac12}]\in (\mathcal{L}_{n+1,\infty}(L_2(\mathbb{R}^{n+1})))_0.$$
This implies that
$$[R_k,M_f]=iA_k+E_0,\quad E_0\in (\mathcal{L}_{n+1,\infty}(L_2(\mathbb{R}^{n+1})))_0,$$
where we denote for brevity
$$A_k=\left(M_{\partial_k f}-\sum_{m=1}^{n+1}M_{\partial_m f}\frac{\partial_k\partial_m}{1+\Delta}\right)(1+\Delta)^{-\frac12}M_{\psi}.$$
Note that $A_k=M_{\phi}A_k=A_kM_{\phi}.$ By triangle inequality,
\begin{align*}
&\Big\|M_{\chi_{Q_j}}A_kM_{\chi_{Q_j}}-M_{\chi_{\mathbb{R}^{n+1}_+}}A_k M_{\chi_{\mathbb{R}^{n+1}_+}}\Big\|_{\mathcal{L}_{n+1,\infty}(L_2(\mathbb{R}^{n+1}))}\\
&\leq 2\Big\|M_{\chi_{\mathbb{R}^{n+1}_+\backslash Q_j}}A_k\Big\|_{\mathcal{L}_{n+1,\infty}(L_2(\mathbb{R}^{n+1}))}+2\Big\|A_kM_{\chi_{\mathbb{R}^{n+1}_+\backslash Q_j}}\Big\|_{\mathcal{L}_{n+1,\infty}(L_2(\mathbb{R}^{n+1}))}\\
&=2\Big\|M_{\chi_{\mathbb{R}^{n+1}_+\backslash Q_j}\cdot\phi}A_k\Big\|_{\mathcal{L}_{n+1,\infty}(L_2(\mathbb{R}^{n+1}))}+2\Big\|A_kM_{\phi\cdot\chi_{\mathbb{R}^{n+1}_+\backslash Q_j}}\Big\|_{\mathcal{L}_{n+1,\infty}(L_2(\mathbb{R}^{n+1}))}\\
&\leq 2\Big\|M_{\chi_{\mathbb{R}^{n+1}_+\backslash Q_j}\cdot\phi}(1+\Delta)^{-\frac12}\Big\|_{\mathcal{L}_{n+1,\infty}(L_2(\mathbb{R}^{n+1}))}\cdot \Big\|(1+\Delta)^{\frac12}A_k\Big\|_{\mathcal{L}_{\infty}(L_2(\mathbb{R}^{n+1}))}\\
&\hspace{1.0cm}+2\Big\|A_k(1+\Delta)^{\frac12}\Big\|_{\mathcal{L}_{\infty}(L_2(\mathbb{R}^{n+1}))}\cdot \Big\|(1+\Delta)^{-\frac12}M_{\phi\cdot\chi_{\mathbb{R}^{n+1}_+\backslash Q_j}}\Big\|_{\mathcal{L}_{n+1,\infty}(L_2(\mathbb{R}^{n+1}))}.
\end{align*}
The assertion follows now from Lemma \ref{ii pre-verification lemma}.
\end{proof}

The following proposition provides a suitable global approximation of the Bessel--Riesz commutators.
\begin{proposition}\label{first approximate lemma} Let $1\leq k\leq n+1$ and $f\in C^{\infty}_c(\mathbb{R}^{n+1}),$ then
$$[R_{\lambda,k},M_{E^{\ast}f}]-\kappa_{n,\lambda}^{[3]}F_{2,0}(0)\Big( M_{x_{n+1}^{-\lambda}} E^{\ast}[R_k,M_f]E M_{x_{n+1}^{\lambda}}\Big)\in (\mathcal{L}_{n+1,\infty}(L_2(\mathbb{R}^{n+1}_+,m_{\lambda})))_0.$$
\end{proposition}
\begin{proof} Recall from Proposition \ref{commutator representation lemma} that
\begin{align*}
[R_{\lambda,k},M_{E^{\ast}f}]&=\kappa_{n,\lambda}^{[3]}M_{x_{n+1}^{-\lambda}}\cdot\mathfrak{S}_{F_{2,0}\circ H}\Big(  E^{\ast}[R_k,M_{f}]E\Big)\cdot  M_{x_{n+1}^{\lambda}}\\
&\hspace{-1.8cm}+\kappa_{n,\lambda}^{[3]}\delta_{k,n+1}\sum_{l=1}^{n+1}M_{x_{n+1}^{-\lambda}} \cdot \Big(\mathfrak{S}_{h_l}\circ\mathfrak{S}_a\circ\mathfrak{S}_{F_{1,1}\circ H}\Big)\Big(E^{\ast}[R_l,M_{f}]E \Big)\cdot M_{x_{n+1}^{\lambda}}\\
&\hspace{-1.8cm}-\kappa_{n,\lambda}^{[3]}\delta_{k,n+1}\sum_{l=1}^{n+1} M_{x_{n+1}^{-\lambda}}\cdot\Big(\mathfrak{S}_{h_l}\circ\mathfrak{S}_{h_{n+1}}\circ\mathfrak{S}_b\circ\mathfrak{S}_{F_{2,1}\circ H}\Big)\Big( E^{\ast}[R_l,M_{f}]E\Big)\cdot  M_{x_{n+1}^{\lambda}}.
\end{align*}
Let us verify the conditions in Lemma \ref{criterion lemma} for
$$V=E^{\ast}[R_k,M_f]E,\quad V_j=M_{\chi_{Q_j}}E^{\ast}[R_k,M_f]EM_{\chi_{Q_j}},$$
where $Q_j=[0,j]^n\times[\frac1j,j].$

The condition \eqref{cla} in Lemma \ref{criterion lemma} is established in Lemma \ref{ii verification lemma}.

Let us now verify the condition \eqref{clb} in Lemma \ref{criterion lemma}. Note that
$$[M_{x_l},R_k]=-\delta_{k,l}\Delta^{-\frac12}+\partial_k[M_{x_l},\Delta^{-\frac12}].$$
By taking Fourier transform on both side, we see that $\partial_k[M_{x_l},\Delta^{-\frac12}]=-\partial_k\partial_l\Delta^{-\frac32}.$ Thus,
$$[M_{x_l},R_k]=-\delta_{k,l}\Delta^{-\frac12}-\partial_k\partial_l\Delta^{-\frac32}.$$
Hence,
$$[M_{x_l},V_j]=-\delta_{k,l}\cdot E^{\ast}\cdot M_{\chi_{Q_j}}[\Delta^{-\frac12},M_f]M_{\chi_{Q_j}}\cdot E-E^{\ast}\cdot M_{\chi_{Q_j}}[\partial_k\partial_l\Delta^{-\frac32},M_f]M_{\chi_{Q_j}}\cdot E.$$
The condition \eqref{clb} in Lemma \ref{criterion lemma} follows now from Lemma \ref{iii verification second lemma}.

The condition \eqref{clc} in Lemma \ref{criterion lemma} trivially holds. Applying Lemma \ref{criterion lemma}
and taking into account that $F_{1,1}(0)=F_{2,1}(0)=0$, we obtain
$$\mathfrak{S}_{F_{2,0}\circ H}\Big(E^{\ast}[R_k,M_f]E\Big)-F_{2,0}(0) E^{\ast}[R_k,M_f]E\in (\mathcal{L}_{n+1,\infty}(L_2(\mathbb{R}^{n+1}_+)))_0,$$
$$\mathfrak{S}_{F_{1,1}\circ H}\Big(  E^{\ast}[R_k,M_f]E\Big)\in (\mathcal{L}_{n+1,\infty}(L_2(\mathbb{R}^{n+1}_+)))_0,$$
$$\mathfrak{S}_{F_{2,1}\circ H}\Big(  E^{\ast}[R_k,M_f]E\Big)\in (\mathcal{L}_{n+1,\infty}(L_2(\mathbb{R}^{n+1}_+)))_0.$$
These estimates, in combination with Proposition \ref{standard schur lemma}, yield the proof of Proposition \ref{first approximate lemma}.
\end{proof}

\begin{lemma}\label{from frank lemma} If $f\in C^{\infty}_c(\mathbb{R}^{n+1}),$ then there exists a constant $C_n>0$ such that for every $1\leq k\leq n+1,$ there exists a limit
$$\lim_{t\to\infty}t^{\frac1{n+1}}\mu_{B(L_2(\mathbb{R}^{n+1}))}(t, M_{\chi_{\mathbb{R}^{n+1}_+}}[R_k,M_f]M_{\chi_{\mathbb{R}^{n+1}_+}})=C_n\|f\|_{\dot{W}^{1,n+1}(\mathbb{R}_+^{n+1})}^{(k)}.$$
\end{lemma}
\begin{proof} The argument repeats that in \cite[Proposition 8.6]{FSZ} {\it mutatis mutandi}.
\end{proof}

The following is a well known approximation lemma due to Birman and Solomyak \cite[Section 11.6]{MR1192782}.

\begin{lemma}\label{approlem}
Let $0<p<\infty,$ $H$ be a Hilbert space and $(A_k)_{k\geq 1}\subset \mathcal{L}_{p,\infty}(H)$ be such that
\begin{enumerate}[{\rm (i)}]
\item $A_k\rightarrow A$ in $\mathcal{L}_{p,\infty}(H)$;
\item for every $k\geq 1$, the limit
$$\lim_{t\rightarrow \infty}t^{\frac{1}{p}}\mu_{B(H)}(t,A_k)=c_k,\quad {\rm exists},$$
\end{enumerate}
then the following limits exist and are equal:
$$\lim_{t\rightarrow \infty}t^{\frac{1}{p}}\mu_{B(H)}(t,A)=\lim_{k\rightarrow \infty}c_k.$$
\end{lemma}

\begin{proof}[Proof of Theorem \ref{mtc}]
We first show the assertion under an extra assumption that $f=E^{\ast}h,$ $h\in C_c^\infty(\mathbb{R}^{n+1}).$ To this end, it follows from Proposition \ref{first approximate lemma} that
\begin{align}\label{spectralasym1}
&\lim_{t\to\infty}t^{\frac1{n+1}}\mu_{B(L_2(\mathbb{R}^{n+1}_+,m_{\lambda}))}(t,[R_{\lambda,k},M_f])\nonumber\\
&\stackrel{P.\ref{first approximate lemma}}{=}\kappa_{n,\lambda}^{[3]}F_{2,0}(0)\lim_{t\to\infty}t^{\frac1{n+1}}\mu_{B(L_2(\mathbb{R}^{n+1}_+,m_{\lambda}))}(t, M_{x_{n+1}^{-\lambda}} E^{\ast}[R_k,M_h]E M_{x_{n+1}^{\lambda}})\nonumber\\
&\stackrel{L.\ref{weight}}{=}\kappa_{n,\lambda}^{[3]}F_{2,0}(0)\lim_{t\to\infty}t^{\frac1{n+1}}\mu_{B(L_2(\mathbb{R}^{n+1}_+))}(t, E^{\ast}[R_k,M_h]E )\nonumber\\
&\stackrel{L.\ref{half}}{=}\kappa_{n,\lambda}^{[3]}F_{2,0}(0)\lim_{t\to\infty}t^{\frac1{n+1}}\mu_{B(L_2(\mathbb{R}^{n+1}))}(t, M_{\chi_{\mathbb{R}^{n+1}_+}}[R_k,M_h]M_{\chi_{\mathbb{R}^{n+1}_+}}).
\end{align}
This should be understood as follows: if one of the limits above exists, then so do the others and the equality holds. It follows from Lemma \ref{from frank lemma} that the latter limit actually exists and equals $C_n\|f\|_{\dot{W}^{1,n+1}(\mathbb{R}_+^{n+1})}^{(k)}$ for some $C_n>0$.  This completes the proof under the extra assumption made above.

Next we apply an approximation argument to remove the extra assumption. To this end, we suppose $f\in  \dot{W}^{1,n+1}(\mathbb{R}_+^{n+1})\cap L_{\infty}(\mathbb{R}_+^{n+1})$ and let $\{f_m\}_{m\geq 1}$ be the sequence chosen in Lemma \ref{density lemma}. Since $E^{\ast}f_m\to f$ in $\dot{W}^{1,n+1}(\mathbb{R}^{n+1}_+),$ it follows from the upper bound in Theorem \ref{main theorem} that
$[R_{\lambda,k},M_{E^{\ast}f_m}]\rightarrow [R_{\lambda,k},M_f]$ in $\mathcal{L}_{n+1,\infty}(L_2(\mathbb{R}^{n+1}_+,m_{\lambda})).$ The assertion follows now from the preceding paragraph and Lemma \ref{approlem}.
\end{proof}

\appendix

\section{Smoothness of auxiliary functions}
\setcounter{equation}{0}

Set
$$G_{k,l}(x):=\frac{F_{k,l}(x)-F_{k,l}(0)}{x},\quad x\in(0,\infty),$$
where $F_{k,l}(0)$ is simply denoted by the right limit of $F_{k,l}$ at $0$. The existence of this limit can be seen in this Appendix.
\begin{theorem}\label{fkl are smooth theorem} If $n\in\mathbb{N}$ and if $(k,l)\in\{(2,0),(1,1),(2,1)\},$ then the functions $F_{k,l}$ and $G_{k,l}$ satisfy the conditions in Theorem \ref{main schur theorem}.
\end{theorem}

\begin{lemma}\label{first fkl lemma} Let $\psi\in L_{\infty}(0,1).$ Then the function
$$d(x):=x^n\int_0^1(x^2+t)^{-\lambda-\frac{n}{2}-1}t^{\lambda+n+2}\psi(t)dt$$
satisfies $d^{(j)}\in L_{\infty}(0,1)$ for $0\leq j\leq n+2.$
\end{lemma}
\begin{proof} We write
$$d^{(j)}(x)=\sum_{\substack{j_1,j_2\geq0\\ j_1+j_2=j}}c(n,j_1,j_2)x^{n+j_2-j_1}\int_0^1(x^2+t)^{-\lambda-\frac{n}{2}-1-j_2}t^{\lambda+n+2}\psi(t)dt.$$
Note that $c(n,j_1,j_2)=0$ if $n+j_2-j_1<0$ and that for any $j_1,j_2\geq0$ with $n+j_2-j_1\geq 0$ and $j=j_1+j_2\leq n+2$, we have
$$x^{n+j_2-j_1}(x^2+t)^{-\lambda-\frac{n}{2}-1-j_2}t^{\lambda+n+2}\leq 1.$$
Thus,
$$|d^{(j)}(x)|\leq \sum_{\substack{j_1,j_2\geq0\\ j_1+j_2=j}}|c(n,j_1,j_2)|\cdot \|\psi\|_{\infty}.$$
This yields the assertion.
\end{proof}

\begin{lemma}\label{second fkl lemma} For $(k,l)\in\{(2,0),(1,1),(2,1)\},$ we have
$$F_{k,l}(x)=x^{n+k}A_l(x)+x^kB_l(x)+x^{k+2l-2}\sum_{j=0}^{n+2}\frac{(-1)^j}{2^{l+2j+1}}\binom{\lambda-1}{j}C_{l,j}(x),$$
where
$$A_l(x):=\int_{\frac12}^2(x^2+2t)^{-\lambda-\frac{n}{2}-1}(2t-t^2)^{\lambda-1}t^ldt,$$
$$B_l(x):=x^n\int_0^{\frac12}(x^2+2t)^{-\lambda-\frac{n}{2}-1}(2t)^{\lambda-1}\Big((1-\frac{t}{2})^{\lambda-1}-\sum_{j=0}^{n+2}\binom{\lambda-1}{j}(-\frac{t}{2})^j\Big)t^ldt,$$
$$C_{l,j}(x):=x^{2j}\int_{x^2}^{\infty}(s+1)^{-\lambda-\frac{n}{2}-1}s^{\frac{n}{2}-j-l}ds.$$
\end{lemma}
\begin{proof} It is immediate that
$$F_{k,l}(x)=x^{n+k}A_l(x)+x^kB_l(x)+D_{k,l}(x),$$
where
$$D_{k,l}(x):=x^{n+k}\sum_{j=0}^{n+2}\frac{(-1)^j}{2^{l+2j}}\binom{\lambda-1}{j}\int_0^{\frac12}(x^2+2t)^{-\lambda-\frac{n}{2}-1}(2t)^{\lambda+l+j-1}dt.$$
Substituting $t=\frac{x^2}{2}s^{-1},$ we conclude that
$$D_{k,l}(x)=x^{k+2l-2}\sum_{j=0}^{n+2}\frac{(-1)^j}{2^{l+2j+1}}\binom{\lambda-1}{j}C_{l,j}(x).$$
This completes the proof.
\end{proof}

\begin{lemma}\label{third fkl lemma} Assume that $(k,l)\in\{(2,0),(1,1),(2,1)\}$. If $n\in\mathbb{N}$ is odd, then $C_{l,j}$ is real analytic near $0.$ If $n\in\mathbb{N}$ is even, then $x\to C_{l,j}(x)-a_{l,j}x^{2j}\log(x)$ is real analytic near $0$ for some constant $a_{l,j}$ which vanishes if $j+l<\frac{n}{2}+1.$
\end{lemma}
\begin{proof} We write
\begin{align}\label{clj}
C_{l,j}(x)=x^{2j}\int_{x^2}^{\frac12}(s+1)^{-\lambda-\frac{n}{2}-1}s^{\frac{n}{2}-j-l}ds+x^{2j}\int_{\frac12}^{\infty}(s+1)^{-\lambda-\frac{n}{2}-1}s^{\frac{n}{2}-j-l}ds.
\end{align}
By Taylor's expansion,
$$(s+1)^{-\lambda-\frac{n}{2}-1}=\sum_{m\geq0}\binom{-\lambda-\frac{n}{2}-1}{m}s^m,$$
where the series converges uniformly for $s\in[0,\frac12].$ Substituting the above equality into \eqref{clj}, we deduce that
$$C_{l,j}(x)=x^{2j}\int_{\frac12}^{\infty}(s+1)^{-\lambda-\frac{n}{2}-1}s^{\frac{n}{2}-j-l}ds+x^{2j}\sum_{m\geq0}\binom{-\lambda-\frac{n}{2}-1}{m}\int_{x^2}^{\frac12}s^{\frac{n}{2}+m-j-l}ds.$$

{\bf Case 1:  $n$ is odd.}
\begin{align*}
C_{l,j}(x)&=x^{2j}\cdot\Big( \int_{\frac12}^{\infty}(s+1)^{-\lambda-\frac{n}{2}-1}s^{\frac{n}{2}-j-l}ds+\sum_{m\geq0}\binom{-\lambda-\frac{n}{2}-1}{m}\frac{2^{j-m+l-\frac{n}{2}-1}}{\frac{n}{2}+m-j-l+1}\Big)\\
&\hspace{1.0cm}-\sum_{m\geq0}\binom{-\lambda-\frac{n}{2}-1}{m}\frac{x^{n+2m-2l+2}}{\frac{n}{2}+m-j-l+1}.
\end{align*}
Thus, $C_{l,j}$ is real analytic near $0.$

{\bf Case 2:  $n$ is even.}
\begin{align*}
&C_{l,j}(x)=x^{2j}\int_{\frac12}^{\infty}(s+1)^{-\lambda-\frac{n}{2}-1}s^{\frac{n}{2}-j-l}ds+x^{2j}\sum_{m\geq0}\binom{-\lambda-\frac{n}{2}-1}{m}\int_{x^2}^{\frac12}s^{\frac{n}{2}+m-j-l}ds\\
&=x^{2j}\cdot\Big( \int_{\frac12}^{\infty}(s+1)^{-\lambda-\frac{n}{2}-1}s^{\frac{n}{2}-j-l}ds+\sum_{\substack{m\geq0\\ m\neq j+l-\frac{n}{2}-1}}\binom{-\lambda-\frac{n}{2}-1}{m}\frac{2^{j-m+l-\frac{n}{2}-1}}{\frac{n}{2}+m-j-l+1}\Big)\\
&\hspace{1.0cm}-\sum_{\substack{m\geq0\\ m\neq j+l-\frac{n}{2}-1}}\binom{-\lambda-\frac{n}{2}-1}{m}\frac{x^{n+2m-2l+2}}{\frac{n}{2}+m-j-l+1}\\
&\hspace{1.0cm}-x^{2j}\cdot \binom{-\lambda-\frac{n}{2}-1}{j+l-\frac{n}{2}-1}\cdot\chi_{\mathbb{N}}(j+l-\frac{n}{2}-1)\cdot (\log(2)+2\log(x)).
\end{align*}
Set $$a_{l,j}:=2\binom{-\lambda-\frac{n}{2}-1}{j+l-\frac{n}{2}-1}\chi_{\mathbb{N}}(j+l-\frac{n}{2}-1).$$
Thus, $C_{l,j}(x)-a_{l,j}x^{2j}\log(x)$ is real analytic near $0.$
\end{proof}

\begin{lemma}\label{fourth fkl lemma} Let $(k,l)\in\{(2,0),(1,1),(2,1)\},$ then for any $j\geq0$, there is a constant $C_{n,\lambda,l,j}>0$ such that
$$|F_{k,l}^{(j)}(x)|\leq C_{n,\lambda,l,j}x^{k-2-2\lambda-j},\quad x\in [0,\infty).$$
\end{lemma}
\begin{proof}
By Leibniz's rule,
$$F_{k,l}^{(j)}(x)=\sum_{\substack{j_1+j_2\geq0\\ j_1+j_2=j}}c(n,j_1,j_2)x^{n+k+j_2-j_1}\int_0^2(x^2+2t)^{-\lambda-\frac{n}{2}-1-j_2}(2t-t^2)^{\lambda-1}t^ldt.$$
Clearly,
$$x^{n+k+j_2-j_1}(x^2+2t)^{-\lambda-\frac{n}{2}-1-j_2}\leq x^{k-2-2\lambda-j_1-j_2}.$$
Thus,
$$|F_{k,l}^{(j)}(x)|\leq\sum_{\substack{j_1+j_2\geq0\\ j_1+j_2=j}}|c(n,j_1,j_2)|x^{k-2-2\lambda-j}\int_0^2(2t-t^2)^{\lambda-1}t^ldt.$$
This implies the assertion.
\end{proof}

\begin{lemma}\label{fifth fkl lemma} Let $(k,l)\in\{(2,0),(1,1),(2,1)\}.$ There exists a polynomial $P_{k,l}$ such that $P_{k,l}(0)=P_{k,l}'(0)=0$ and such that the mapping $x\to F_{k,l}(x)-P_{k,l}(x)\log(x)$ belongs to $C^{n+2}[0,1].$
\end{lemma}
\begin{proof} We use the decomposition
$$F_{k,l}(x)=x^{n+k}A_l(x)+x^kB_l(x)+x^{k+2l-2}\sum_{j=0}^{n+2}\frac{(-1)^j}{2^{l+2j+1}}\binom{\lambda-1}{j}C_{l,j}(x)$$
given in Lemma \ref{second fkl lemma}.

Let $a_{l,j}$ be the constant given in Lemma \ref{third fkl lemma} (we set $a_{l,j}=0$ for odd $n$). Set
$$P_{k,l}(x)=\sum_{j=0}^{n+2}\frac{(-1)^j}{2^{l+2j+1}}\binom{\lambda-1}{j}a_{l,j}x^{(k+2l-2)+2j}$$
and note that $P_{k,l}(0)=P_{k,l}'(0)=0$ by Lemma \ref{third fkl lemma}. We have
\begin{align*}
F_{k,l}(x)-P_{k,l}(x)\log(x)&=x^{n+k}A_l(x)+x^kB_l(x)\\
&\hspace{-1.0cm}+x^{k+2l-2}\sum_{j=0}^{n+2}\frac{(-1)^j}{2^{l+2j+1}}\binom{\lambda-1}{j}\big(C_{l,j}(x)-a_{l,j}x^{2j}\log(x)\big).
\end{align*}

Clearly, the first summand is infinitely differentiable. Next, we recall from Lemma \ref{second fkl lemma} that
$$B_l(x)=x^n\int_0^1(x^2+t)^{-\lambda-\frac{n}{2}-1}t^{\lambda+n+2}\psi_l(t)dt,$$
where
$$\psi_l(t):=2^{-l-1}t^{l-n-3}\Big((1-\frac{t}{4})^{\lambda-1}-\sum_{j=0}^{n+2}\binom{\lambda-1}{j}(-\frac{t}{4})^j\Big),\quad t\in(0,1).$$
Since $\psi_l\in L_{\infty}(0,1),$ it follows from Lemma \ref{first fkl lemma} that $B_l\in C^{n+2}[0,1].$ Hence, the second summand in the decomposition is  $n+2$ times differentiable.
 The last summands are infinitely differentiable by Lemma \ref{third fkl lemma}. This completes the proof.
\end{proof}

\begin{proof}[Proof of Theorem \ref{fkl are smooth theorem}]
It is obvious that each $F_{k,l}$ and $G_{k,l}$ is smooth on $(0,\infty).$ It follows from Lemma \ref{fourth fkl lemma} that $F_{k,l}\circ\exp\in W^{1+\lceil\frac{n+1}{2}\rceil,2}(\mathbb{R}_+).$ Thus, also $G_{k,l}\circ\exp\in W^{1+\lceil\frac{n+1}{2}\rceil,2}(\mathbb{R}_+).$

Let $P_{k,l}$ be the polynomial constructed in Lemma \ref{fifth fkl lemma}. Write $G_{k,l}=Q_{k,l}+R_{k,l},$ where
$$Q_{k,l}(x)=\frac{P_{k,l}(x)\log(x)}{x},\quad x\in [0,1],$$
$$R_{k,l}(x)=\frac{F_{k,l}(x)-P_{k,l}(x)\log(x)-F_{k,l}(0)}{x},\quad x\in [0,1].$$

By Lemma \ref{fifth fkl lemma}, $R_{k,l}\in C^{n+1}[0,1]\subset C^{1+\lceil\frac{n+1}{2}\rceil}[0,1].$ Thus, $(R_{k,l}-R_{k,l}(0))\circ\exp\in W^{1+\lceil\frac{n+1}{2}\rceil,2}(\mathbb{R}_-).$ It is immediate that $Q_{k,l}\circ\exp\in W^{1+\lceil\frac{n+1}{2}\rceil,2}(\mathbb{R}_-).$ Consequently, $(G_{k,l}-G_{k,l}(0))\circ\exp\in W^{1+\lceil\frac{n+1}{2}\rceil,2}(\mathbb{R}_-).$ This immediately yields $(F_{k,l}-F_{k,l}(0))\circ\exp\in W^{1+\lceil\frac{n+1}{2}\rceil,2}(\mathbb{R}_-).$ This ends the proof of Theorem \ref{fkl are smooth theorem}.
\end{proof}

\bigskip
\bigskip
{\bf Acknowledgements:}

Z. Fan is supported by the Guangdong Basic and Applied Basic Research Foundation (No. 2023A1515110879), by the China Postdoctoral Science Foundation (No. 2023M740799) and by the Postdoctoral Fellowship Program of CPSF (No. GZB20230175). J. Li is supported by the Australian Research Council (ARC) through the research grant DP220100285. F. Sukochev and D. Zanin are supported by the Australian Research Council through the research grant DP230100434. F. Sukochev is supported by the Australian Research Council Laureate Fellowship FL170100052. Z. Fan would like to thank Prof. Dongyong Yang for discussions about harmonic analysis associated with Bessel operator.

\end{document}